\documentclass[final]{siamltex}

\DeclareMathAlphabet{\mathpzc}{OT1}{pzc}{m}{it}
\newcommand{\bff}{\mathpzc{f}}
\newcommand{\mcA}{\mathcal{A}}
\newcommand{\mcB}{\mathcal{B}}
\newcommand{\mcF}{\mathcal{F}}
\newcommand{\mcL}{\mathcal{L}}
\newcommand{\ul}{\underline}

\newcommand{\citecond}[1]{\textnormal{\textbf{#1}}}
\newcommand{\hnorm}{|\hspace{-.1em}|\hspace{-.1em}|}
\newcommand{\nnsup}[2]{\sup_{0\not=#1\in#2}}
\newcommand{\les}{\lesssim}
\newcommand{\ges}{\gtrsim}

\usepackage[english]{babel}
\usepackage[latin1]{inputenc}
\usepackage[T1]{fontenc}
\usepackage{amssymb,amsmath}
\usepackage{enumitem}
\usepackage{graphicx,url}

\usepackage{color,ulem}

\newtheorem{remark}[theorem]{Remark}
\newtheorem{notation}[theorem]{Notation}

\title{A robust multigrid method for the time-dependent Stokes 
	problem\thanks{The research was funded by the Austrian Science Fund (FWF):
	J3362-N25.}}

\author{Stefan Takacs\thanks{Faculty of Mathematics,
			TU Chemnitz, Germany, ({\tt stefan.takacs@numa.uni-linz.ac.at})}}

\begin{document}

\maketitle

\begin{abstract}
We propose a coupled multigrid method
for generalized Stokes flow problems. Such problems occur as subproblems
in implicit time-stepping approaches for time-dependent Stokes problems. 
The discretized Stokes system is a large-scale linear system whose condition
number depends on the grid size of
the spatial discretization and of the length of the time step. Recently,
for this problem a coupled multigrid method has been proposed, where
in each smoothing step a Poisson problem has to be solved (approximately)
for the pressure field.
In the present paper, we propose a coupled multigrid method where
the solution of such sub-problems is not needed.
We prove that the proposed method shows robust convergence behavior
in the grid size of the spatial discretization and of the length 
of the time-step.
\end{abstract}

\begin{keywords} 
Generalized Stokes problem, coupled multigrid methods, robustness
\end{keywords}

\begin{AMS}
65N55 65N22 76D07 65N30
\end{AMS}

\pagestyle{myheadings}
\thispagestyle{plain}
\markboth{S. TAKACS}{MULTIGRID FOR TIME-DEPENDENT STOKES PROBLEM}

\normalem

\section{Introduction}\label{sec:1}

We consider the following model
problem \emph{(generalized Stokes flow problem}).
Let $\Omega\subset\mathbb{R}^2$ be a bounded polygonal domain and
assume $f \in [L^2(\Omega)]^2$ and $g \in L^2(\Omega)$ to be given.
Find a velocity field $u$ and a pressure distribution $p$ such that
\begin{equation}\label{eq:stokes}
\begin{aligned}
        -\Delta u + \beta u + \nabla p &= f \mbox{ in } \Omega,  \\
        \nabla \cdot u &= g \mbox{ in } \Omega, \\
        u &= 0
        \mbox{ on } \partial \Omega \\
\end{aligned}
\end{equation}
is satisfied. $\beta>0$ is assumed to be a given parameter. 
The problem~\eqref{eq:stokes} appears as an auxiliary problem for
implicit time-stepping approaches to solve an incompressible, time-dependent
Stokes flow problem. In this case, the parameter~$\beta$ is proportional
to the inverse of the length of the time-step, scaled by a viscosity parameter.

To obtain existence and uniqueness of the solution, we further require 
$\int_{\Omega} p \,\mbox{d}x = \int_{\Omega} g \,\mbox{d}x =0$. Note that
the analysis presented in this paper is (due to the use of regularity results)
restricted to convex domains. However, the numerical method presented in
the paper can be applied also to non-convex domains.

The discretization of the problem leads to an indefinite linear system with 
saddle-point structure. The main goal of this work is to construct and to 
analyze numerical methods that produce an approximate solution
to the problem, where the computational complexity can be
bounded by the number of unknowns times a constant which is independent of 
the grid level (of the spatial discretization)
and the choice of~$\beta$, in particular for large values of~$\beta$ (which correspond to small time steps).

For the solution of such a saddle-point problem, there are several possibilities.
In~\cite{Bramble:Pasciak:1997,Kobelkov:Olshanskii:2000,Mardal:Winther:2004,
Mardal:Winther:2004err,Olshanskii:Peters:Reusken:2006,Zulehner:2010}, various kinds of
preconditioners have been proposed for this problem which can
be combined with a Krylov-subspace method as outer iteration scheme to yield an 
iterative solver for the problem.

An alternative is to apply a multigrid algorithm directly to the coupled system.
On the one hand such methods are typically quite fast, on the other hand such an approach
does not require an outer iteration scheme.
For~$\beta=0$, problem~\eqref{eq:stokes} is the standard (stationary) Stokes problem.
For this case several multigrid solvers are available, see, 
e.g.,~\cite{Verfuehrth:1984,Vanka:1986,Brenner:1990,Braess:Sarazin:1997,Zulehner:2000} and the papers 
cited in~\cite{Wesseling:Oosterlee:2001,Olshanskii:2012}. The construction of a 
multigrid method for $\beta>0$, particularly if
the method is desired to show robust convergence behavior in~$\beta$, is more involved,
see~\cite{Larin:Reusken:2008} for an overview and numerical results.
Recently, an important step forward has been archived by Olshanskii, who has
proposed such a robust multigrid method, see~\cite{Olshanskii:2012}.
In the named paper however, a Poisson problem for the pressure has to be solved
approximately (for example by applying one V-cycle, cf. Section~8 in~\cite{Olshanskii:2012}) 
for each step of the smoothing iteration. (We will comment on this in Remarks~\ref{rem:on:smoothers1}
and~\ref{rem:on:smoothers2}).

The goal of the present paper is to drop this requirement and, therefore, reduce
the computational costs. We propose a multigrid solver where the smoother is a \emph{simple}
linear iteration scheme and we prove that the proposed method is robust in the
grid size of the spatial discretization and in the choice of~$\beta$. We will
present a convergence proof for our multigrid method
based on the classical splitting of the analysis
into smoothing property and approximation property, see~\cite{Hackbusch:1985}.

The results of this paper form a basis of the convergence analysis of an all-at-once
multigrid method for the Stokes \emph{optimal control} problem, cf.~\cite{Takacs:2013a}.

This paper is organized as follows. In Section~\ref{sec:2} we will
introduce the variational formulation and discuss its discretization.
The multigrid framework will be presented in Section~\ref{sec:3}.
In Section~\ref{sec:3aa} we will discuss the choice of the smoother.
The proof of the approximation property will be given in Section~\ref{sec:3a}.
Numerical results which illustrate the convergence result will be presented
in Section~\ref{sec:4}. In Section~\ref{sec:5} we will close with conclusions.

\section{Variational formulation and discretization}\label{sec:2}

Here and in what follows,
$L^2(\Omega)$ and~$H^1(\Omega)$ denote the standard Lebesgue and
Sobolev spaces with associated standard
norms~$\|\cdot\|_{L^2(\Omega)}$ and~$\|\cdot\|_{H^1(\Omega)}$,
respectively. The space $L^2_0(\Omega)$ is the space of all functions in $L^2(\Omega)$ with
mean value $0$.
The space $H^1_0(\Omega)$ is the space of all functions in $H^1(\Omega)$ that vanish on the
boundary. Both spaces are equipped with  standard norms, i.e.,~${\|\cdot\|_{L^2_0(\Omega)}}:={\|\cdot\|_{L^2(\Omega)}}$
and~${\|\cdot\|_{H^1_0(\Omega)}}:={\|\cdot\|_{H^1(\Omega)}}$.

Using these spaces, we can set up the variational formulation of~\eqref{eq:stokes},
which reads as follows. Find $u\in U:= [H^1_0(\Omega)]^2$ and $p\in P:= L^2_0(\Omega)$ such that
\begin{align*}
		(\nabla u,\nabla\tilde{u})_{L^2(\Omega)} 
		+\beta( u,\tilde{u})_{L^2(\Omega)} 
		+(p,\nabla\cdot\tilde{u})_{L^2(\Omega)}
		&=(f,\tilde{u})_{L^2(\Omega)}\\
          (\nabla\cdot u,\tilde{p})_{L^2(\Omega)} 
		&= (g,\tilde{p})_{L^2(\Omega)}
\end{align*}
holds for all $\tilde{u}\in U$ and $\tilde{p}\in P$.
Certainly, the variational problem can be rewritten as 
one variational equation
as follows. Find $x \in X$ such that
\begin{equation}\label{eq:gen:bil}
        \mcB(x,\tilde{x}) = \mcF(\tilde{x}) \qquad \mbox{for all }\tilde{x} \in X,
\end{equation}
where $X:=U\times P$, $x:=(u,p)$, $\tilde{x}:=(\tilde{u},\tilde{p})$  and
\begin{align*}\nonumber
        \mcB((u,p),(\tilde{u},\tilde{p})) &:=
		(\nabla u,\nabla\tilde{u})_{L^2(\Omega)} 
		+\beta ( u,\tilde{u})_{L^2(\Omega)} 
		+(p,\nabla\cdot\tilde{u})_{L^2(\Omega)}
		+(\nabla\cdot u,\tilde{p})_{L^2(\Omega)} ,
	\\
	        \mcF(\tilde{u},\tilde{p}) &:= (f,\tilde{u})_{L^2(\Omega)}+(g,\tilde{p})_{L^2(\Omega)}.
\end{align*}

We are interested in finding an approximative solution for equation~\eqref{eq:gen:bil}. 
The convergence analysis follows standard approaches, i.e., we show that the problem in question is
well posed in some norm $\|\cdot\|_X$ (which is for the Poisson problem the $H^1$-norm and for the
standard Stokes problem the pair of $H^1$ for the velocity and $L^2$ for the pressure). In our case
this norm will be parameter-dependent. In a second step, we will introduce a properly scaled
$L^2$-like norm $\hnorm\cdot\hnorm_{k}$, which depends on the grid level $k$, see below. We will use
an inverse inequality to show that $\|\cdot\|_X\le C \hnorm\cdot\hnorm_{k}$ holds. The convergence of the multigrid
method (smoothing property and approximation property) will be shown in the norm $\hnorm\cdot\hnorm_{k}$.

The main focus of the paper is the proof of the approximation property. Throughout the paper we
will comment on the differences to the approach presented in~\cite{Olshanskii:2012}, which allows
us to drop the requirement of solving (approximately) Poisson problems. The key idea of the 
proof of the approximation property follows classical approaches. However, the way of constructing
the norms is non-standard. For simplicity, we follow the abstract framework introduced
in~\cite{Takacs:Zulehner:2012}.
First we introduce the following convenient notation.
\begin{notation}
	Throughout this paper,~$C>0$ is a generic constant, independent of
	the grid level~$k$ and the choice of the parameter~$\beta$.
	For any scalars~$a$ and~$b$, we write~$a \les b$ (or $b \ges a$) if there is a 
	constant~$C>0$ such that~$a \le C\,b$. We write~$a\eqsim b$ if~$a\les b\les a$.
\end{notation}

Let the Hilbert spaces $X$, $U$ and $P$ (introduced above) be equipped
with the following norms:
\begin{equation}\label{eq:x:norm}
\begin{aligned}
        \|x\|_X^2  &:=\|(u,p)\|_X^2 := \|u\|_U^2 + \|p\|_{P}^2,\\
      		\|u\|_U^2 &:=  \|u\|_{H^1(\Omega)}^2+\beta \|u\|_{L^2(\Omega)}^2\mbox{ and}\\
       	\|p\|_P^2 &:=  \nnsup{w}{[H^1_0(\Omega)]^2} \frac{(p,\nabla \cdot
		w)_{L^2(\Omega)}^2}{\|w\|_{H^1(\Omega)}^2+\beta \|w\|_{L^2(\Omega)}^2}.
\end{aligned}
\end{equation}
Lemma~2.1 in~\cite{Olshanskii:2012} states the following stability result.
\begin{lemma}
        The relation
		\begin{align*}
			\tag{\textbf{A1}}
			\|x\|_{X} \les \nnsup{\tilde{x}}{X}
				\frac{\mcB(x,\tilde{x})}{\|\tilde{x}\|_{X}}
				\les \|x\|_{X}
		\end{align*}
		holds for all $x\in X$.
\end{lemma}

Using the following notation, we can express the norms in a nicer way.
\begin{notation}
	For any Hilbert space~$A$, the term $A^*$ denotes its dual space equipped
	with the dual norm
        \begin{equation*}
                \|u\|_{A^*} := \nnsup{w}{A}\frac{\langle u,w \rangle}{\|w\|_A},
        \end{equation*}
	where $\langle u,\cdot\rangle := u(\cdot)$ denotes the duality pairing.

	For any Hilbert space~$A$ and any scalar $\alpha>0$, the term~$\alpha\, A$ 
	denotes the space on the underlying set of the Hilbert space~$A$ equipped with the norm
        \begin{equation*}
                \|u\|_{\alpha\,A}^2 := \alpha \|u\|_A^2.
       \end{equation*}
        For any two Hilbert spaces~$A$ and~$B$, the term
        $A\cap B$ denotes the space on the intersection of the underlying sets, 
	$\{u\in A \cap B\}$, equipped with the norm
        \begin{equation*} 
                \|u\|_{A\cap B}^2 := \|u\|_A^2 + \|u\|_B^2
	\end{equation*}
	and the term~$A+B$ denotes the space on the algebraic sum of the
	underlying sets, $\{u_1+u_2\;:\; u_1\in A, u_2\in B\}$, equipped with the norm
        \begin{equation*}
                \|u\|_{A+B}^2 := \inf_{u_1\in A, u_2\in B, u=u_1+u_2} 
			\|u_1\|_A^2 + \|u_2\|_B^2.
        \end{equation*}
\end{notation}

The spaces $A^*$, $\alpha\, A$, $A\cap B$ and $A+B$ are Hilbert spaces. The fact that $A^*$
is a Hilbert space follows directly from the Riesz representation theorem, see, e.g.,
Theorem~1.2 in~\cite{Adams:Fournier}. The fact that $\alpha\, A$ is a Hilbert space
is obvious and for the latter two see, e.g., Lemma~2.3.1 in~\cite{Bergh:Loefstroem:1976}.

We immediately see that the norm on $U$ can be rewritten as follows:
\begin{equation*}
        \|u\|_U = \|u\|_{H^1(\Omega)\cap \beta L^2(\Omega)}.
\end{equation*}

To reformulate the norm on $P$, we need the following regularity assumption.
	\begin{description}[style=sameline,leftmargin=1.1cm]
	        \item[\citecond{(R)}] \emph{Regularity of the Stokes problem.}
	        Let $f\in [L^2(\Omega)]^2$ and $g\in H^1_0(\Omega)\cap L^2_0(\Omega)$ be arbitrary but fixed and
	        $(u,p)\in [H^1_0(\Omega)]^2\times L^2_0(\Omega)$ be the solution of the
		Stokes problem, i.e., such that
	        \begin{equation*}
	        \begin{array}{lclclcl}
	          (\nabla u,\nabla \tilde{u})_{L^2(\Omega)} &+&(p,\nabla
		\tilde{u})_{L^2(\Omega)}  &=&(f,\tilde{u})_{L^2(\Omega)}\\
	          (\nabla u,\tilde{p})_{L^2(\Omega)} &&& =& (g,\tilde{p})_{L^2(\Omega)} \\
	        \end{array}        
	        \end{equation*}
	        holds for all $(\tilde{u},\tilde{p}) \in [H^1_0(\Omega)]^2\times L^2_0(\Omega)$.
	        Then $(u,p)\in [H^2(\Omega)]^2\times H^1(\Omega)$ and
	        \begin{equation*}
	                \|u\|_{H^2(\Omega)}^2 + \|p\|_{H^1(\Omega)}^2 \les \|f\|_{L^2(\Omega)}^2+\|g\|_{H^1(\Omega)}^2.
	        \end{equation*}
	\end{description}
\begin{lemma}
	The regularity assumption~\citecond{(R)} holds for $\Omega\subset\mathbb{R}^2$ being a convex polygonal domain.
\end{lemma}
\begin{proof}
	Theorem~2 in~\cite{Kellogg:Osborn:1976} states (in the notation of the present paper) that
	provided $f\in [L^2(\Omega)]^2$, $g \in H^1(\Omega)\cap L^2_0(\Omega)$ and $\delta^{-1} g \in L^2(\Omega)$ that
	\begin{equation*}
		\|u\|_{H^2(\Omega)}^2 + \|\nabla p\|_{L^2(\Omega)}^2 \les \|f\|_{L^2(\Omega)}^2 
			+ \|\nabla g\|_{L^2(\Omega)}^2 + \|\delta^{-1} g\|_{L^2(\Omega)}^2,
	\end{equation*}
	is satisfied, where $\delta:\Omega\rightarrow \mathbb{R}$ is the distance to the closest corner of the polygonal
	domain $\Omega$.
	Lemma~2 in~\cite{Kellogg:Osborn:1976} states that
	$\|\delta^{-1} g\|_{L^2(\Omega)} \les \|g\|_{H^1(\Omega)}$ is satisfied for all $g \in H^1_0(\Omega)$.
	Combining these results, we obtain that 
	\begin{equation*}
		\|u\|_{H^2(\Omega)}^2 + \|\nabla p\|_{L^2(\Omega)}^2 \les \|f\|_{L^2(\Omega)}^2
			+ \|g\|_{H^1(\Omega)}^2
	\end{equation*}
	is satisfied for all $f\in [L^2(\Omega)]^2$ and $g \in H^1_0(\Omega)\cap L^2_0(\Omega)$.
	As $p\in L^2_0(\Omega)$ was assumed, Poincar\'e's inequality 
	states that $\|p\|_{H^1(\Omega)} \les \|\nabla p\|_{L^2(\Omega)}$, which finishes the proof.
\end{proof}

Note that we assume that $g$ satisfies
homogeneous Dirichlet boundary conditions. This condition can be weakened but it is not possible to drop
such a condition completely, cf.~\cite{Kellogg:Osborn:1976}.
 
\begin{lemma}\label{lem:pe}
	On domains $\Omega$, where~\citecond{(R)} is satisfied,
	\begin{equation*}
		\|p\|_P \eqsim \|p\|_{ L^2(\Omega) + \beta^{-1} H^1(\Omega)}
	\end{equation*}
	holds for all $p\in L^2_0(\Omega)$.
\end{lemma}

For a proof of this lemma, see Theorem~3.2 in~\cite{Olshanskii:Peters:Reusken:2006} (this
proof only needs \citecond{(R)} for $f\in[L^2(\Omega)]^2$ and $g=0$).

The discretization of problem~\eqref{eq:gen:bil} is done using
standard finite element techniques.
We assume to have for $k=0,1,2,\ldots$ a sequence of grids obtained
by uniform refinement. The grid size on grid level $k$ (length of the largest edge)
is denoted by $h_k$. On each
grid level~$k$, we discretize the problem using the Galerkin approach,
i.e., we have finite dimensional
spaces $X_k \subset X$ and consider the following problem. Find $x_k
\in X_k $ such that
\begin{equation}\label{eq:galerkin}
        \mcB(x_k,\tilde{x}_k) = \mcF(\tilde{x}_k) \qquad \mbox{for all
}\tilde{x}_k \in X_k.
\end{equation}
Following the concept of mixed finite elements, we choose $X_k:=U_k\times P_k$,
where $U_k\subset U$ and $P_k\subset P$ are obtained by a proper discretization, see below.
Using a nodal basis, we can represent this problem in matrix-vector
notation as follows:
\begin{equation}\label{eq:gen:bil:mv}
                \mcA_k\,\ul{x}_k = \ul{\bff}_k,
\end{equation}
where
\begin{equation*}
	\mcA_k = \left(\begin{array}{cc}A_k&B_k^T\\B_k&0\end{array}\right),\qquad 
	\ul{x}_k = \left(\begin{array}{c}\ul{u}_k\\\ul{p}_k\end{array}\right),\qquad 
	\ul{\bff}_k = \left(\begin{array}{c}\ul{f}_k\\\ul{g}_k\end{array}\right)
\end{equation*}
and the matrices $A_k$ and $B_k$ represent the scalar products
$a(u,\tilde{u}):=(\nabla u,\nabla\tilde{u})_{L^2(\Omega)} 
		+\beta ( u,\tilde{u})_{L^2(\Omega)}$ and 
$b(u,\tilde{p}):=(\nabla\cdot u,\tilde{p})_{L^2(\Omega)}$, respectively.

Here and in what follows, any underlined quantity, like~$\ul{x}_k$, is
the representation of the corresponding
non-underlined quantity, here~$x_k$, with respect to a nodal basis of
the corresponding Hilbert space, here~$X_k$.

The next step is to show the discrete stability condition, i.e., that
     \begin{equation}\tag{\textbf{A1a}}
               \|x_k\|_{X} \les \nnsup{\tilde{x}_k}{X_k}
                  \frac{\mcB(x_k,\tilde{x}_k)}{\|\tilde{x}_k\|_{X}}
               \les \|x_k\|_{X}
     \end{equation}
holds for all $x_k\in X_k$.

To guarantee the discrete stability condition, we have to choose a discretization which is stable for
the standard Stokes problem, particularly a discretization that satisfies the weak inf-sup
condition, i.e.,
\begin{equation}\label{eq:cond:s}
	\nnsup{u_k}{U_k} \frac{(\nabla\cdot u_k,p_k)_{L^2(\Omega)}}{\|u_k\|_{L^2(\Omega)} } 
		\ges \|\nabla p_k\|_{L^2(\Omega)}
\end{equation}
should hold for all $p_k\in P_k$. Note that this is a standard condition which guarantees that the
chosen discretization is stable for the Stokes problem.
In~\cite{Bercovier:Pironneau:1979,Verfuerth:1984} it was shown that
condition~\eqref{eq:cond:s} is satisfied for the Taylor-Hood element 
($P1-P2$-element) for polygonal
domains where at least one vertex of each element is located in the interior
of the domain. Here and in what follows we assume that the problem is
discretized with the Taylor-Hood element and that the mesh satisfies
the named condition.
Using the weak-inf sup condition~\eqref{eq:cond:s} we can show the following lemma.
\begin{lemma}
	If the problem is discretized using the Taylor-Hood element,
	condition~\citecond{(A1a)} is satisfied.
\end{lemma}
\begin{proof}
	The estimate~(2.16) in~\cite{Olshanskii:2012} states that
  \begin{equation*}
     \|x_k\|_{X_{1,k}} \les \nnsup{\tilde{x}_k}{X_k}
                  \frac{\mcB(x_k,\tilde{x}_k)}{\|\tilde{x}_k\|_{X_{1,k}}},
 	\end{equation*}
 holds for all $x_k\in X_k$, where 
	\begin{equation}\label{eq:x1:norm}
	\begin{aligned}
        	\|x_k\|_{X_{1,k}}^2 &:=\|(u_k,p_k)\|_{X_{1,k}}^2 := \|u_k\|_{U_{1,k}}^2 + \|p_k\|_{P_{1,k}}^2,\\
      		\|u_k\|_{U_{1,k}}^2 &:=  \|u_k\|_{H^1(\Omega)}^2+\beta \|u_k\|_{L^2(\Omega)}^2\mbox{ and}\\
       		\|p_k\|_{P_{1,k}}^2 &:=  \nnsup{w_k}{U_k} \frac{(p_k,\nabla \cdot w_k)_{L^2(\Omega)}^2}{\|w_k\|_{H^1(\Omega)}^2+\beta \|w_k\|_{L^2(\Omega)}^2}.
	\end{aligned}
	\end{equation}
	Here, the underlying sets of the Hilbert spaces $X_{1,k}$, $U_{1,k}$ and $P_{1,k}$ coincide with those
	of $X_{k}$, $U_{k}$ and $P_{k}$, respectively. Observe $\|\cdot\|_{U_{1,k}}=\|\cdot\|_{U}$. Lemma~2.2 in~\cite{Olshanskii:2012} states 
	$\|\cdot\|_{P_{1,k}}\eqsim \|\cdot\|_{P}$. This shows~\citecond{(A1a)} as boundedness follows directly from~\citecond{(A1)}.
\end{proof}

Note that the grids are obtained by uniform refinement. So, the
discrete subsets are nested, i.e., $U_k \subseteq U_{k+1}$, 
$P_{k} \subseteq P_{k+1}$ and, consequently, $X_{k} \subseteq X_{k+1}$.

\section{A coupled multigrid method}\label{sec:3}

The problem~\eqref{eq:gen:bil:mv} shall be solved using a multigrid method.
Starting from an initial approximation~$\ul{x}^{(0)}_k$,
one iterate of the multigrid method is given by the following two steps:
\begin{itemize}
        \item \emph{Smoothing procedure:} Compute
              \begin{equation} \label{eq:sm:comp}
                   \ul{x}^{(0,m)}_k := \ul{x}^{(0,m-1)}_k + \hat{\mcA}_k^{-1}
                                    \left(\ul{\bff}_k -\mcA_k\;\ul{x}^{(0,m-1)}_k\right)
                                    \qquad \mbox{for } m=1,\ldots,\nu
              \end{equation}
        with $\ul{x}^{(0,0)}_k:=\ul{x}^{(0)}_k$. The choice of
        the smoother (or, in other words, of the preconditioning matrix
$\hat{\mcA}_k^{-1}$) will be discussed below. \vspace{.2cm}
        \item \emph{Coarse-grid correction:}
                \begin{itemize}
                     \item Compute the defect 
                        $\ul{\bff}_k -\mcA_k\;\ul{x}^{(0,\nu)}_k$
                        and restrict it to grid level $k-1$ using
                        an restriction matrix $I_k^{k-1}$:
                        \begin{equation}\nonumber
                              \ul{r}_{k-1}^{(1)} := I_k^{k-1} \left(\ul{\bff}_k -\mcA_k
                              \;\ul{x}^{(0,\nu)}_k\right).
                        \end{equation}
                     \item Determine the update $\ul{z}_{k-1}^{(1)}$ by solving the coarse-grid problem
                        \begin{equation}\label{eq:coarse:grid:problem}
                            \mcA_{k-1} \,\ul{z}_{k-1}^{(1)} =\ul{r}_{k-1}^{(1)}
                        \end{equation}
                        approximately.
                     \item Prolongate $\ul{z}_{k-1}^{(1)}$ to the
                          grid level $k$ using an prolongation 
                          matrix $I^k_{k-1}$ and add
                          the result to the previous iterate:
                          \begin{equation}\nonumber
                               \ul{x}_{k}^{(1)} := \ul{x}^{(0,\nu)}_k +
                                I_{k-1}^k \, \ul{z}_{k-1}^{(1)}.
                          \end{equation}
                \end{itemize}
\end{itemize}
As we have assumed to have nested spaces, the intergrid-transfer
matrices $I_{k-1}^k$ and $I_k^{k-1}$ are chosen in
a canonical way: $I_{k-1}^k $ is the canonical embedding
and the restriction $I_k^{k-1}$ is its transposed, i.e., $I_k^{k-1}=(I_{k-1}^k)^T$.

If the problem on the coarser grid is solved exactly (two-grid
method), the coarse-grid correction is given by
\begin{equation} \label{eq:method:cga}
        \ul{x}_k^{(1)} := \ul{x}_k^{(0,\nu)} +
        I_{k-1}^{k} \, \mcA_{k-1}^{-1} \,  I_{k}^{k-1}
        \left( \ul{\bff}_k - \mcA_k \;\ul{x}_k^{(0,\nu)}\right).
\end{equation}
In practice, the problem~\eqref{eq:coarse:grid:problem} is
approximately solved by applying one step (V-cycle)
or two steps (W-cycle) of the multigrid method, recursively. On
grid level $k=0$, the problem~\eqref{eq:coarse:grid:problem} is 
solved exactly.

To construct a multigrid convergence result based on Hackbusch's
splitting of the analysis into smoothing property and approximation
property, we have to introduce an appropriate framework.

Convergence is shown in the following $L^2$-like norms $\hnorm\cdot\hnorm_k$:
\begin{equation*}
        \hnorm x_k\hnorm_k^2 :=\|\ul{x}_k\|_{\mcL_k}^2
        :=(\mcL_k \ul{x}_k,\ul{x}_k)_{\ell^2},
\end{equation*}
where
\begin{equation}\label{eq:defMcL}
        \mcL_k:=
        \left(
                \begin{array}{cc}
                        (h_k^{-2}+\beta) M_{U,k} \\
                        &h_k^{-2}(\beta+h_k^{-2} )^{-1}M_{P,k}
                \end{array}
        \right),
\end{equation}
the matrices $M_{U,k}$ and $M_{P,k}$ are mass-matrices, representing the $L^2$-inner 
product in $U_k$ and $P_k$, respectively, and $(\cdot,\cdot)_{\ell^2}$ is the Euclidean
scalar product.

The smoothing property and the approximation property, which we will show below, read as follows.
\begin{itemize}
        \item \emph{Smoothing property:}
                \begin{equation} \label{eq:smp}
                         \sup_{\tilde{x}_k\in X_k} 
                         \frac{\mcB \left(x_k^{(0,\nu)}-x_k^*,
                                   \tilde{x}_k\right)}{\hnorm \tilde{x}_k\hnorm_k}
                         \le \eta(\nu) \hnorm x_k^{(0)}-x_k^*\hnorm_k
                \end{equation}
                holds for some function $\eta(\nu)$
                with $\lim_{\nu\rightarrow\infty}\eta(\nu)= 0$. Here and in what follows,
		$x_k^*\in X_k$ is the exact solution of the discretized problem~\eqref{eq:galerkin}.
        \item \emph{Approximation property:}
                \begin{equation} \label{eq:apprp}
                        \hnorm x_k^{(1)}-x_k^*\hnorm_k \le
                        C_A \sup_{\tilde{x}_k\in X_k} 
                        \frac{\mcB\left(
				x_k^{(0,\nu)}-x_k^*,\tilde{x}_k\right)}
				{\hnorm\tilde{x}_k\hnorm_k}
                \end{equation}        
                holds for some constant $C_A>0$.
\end{itemize}
If we combine both conditions, we obtain
\begin{equation}\nonumber
                        \hnorm x_k^{(1)}-x_k^*\hnorm_k \le
                        q(\nu) \hnorm x_k^{(0)}-x_k^*\hnorm_k,
\end{equation}
where $q(\nu)=C_A\eta(\nu)$. For $\nu$ large enough, we obtain $q(\nu)<1$,
i.e., the convergence of the two-grid method.
The convergence of the W-cycle multigrid method can be shown under
mild assumptions, see e.g.~\cite{Hackbusch:1985}.

\begin{remark}\label{rem:on:norms1}
	The norm $\hnorm\cdot\hnorm_{k}$ is the discrete analog of 
	$X_{0,k}:=U_{0,k}\times P_{0,k}:= U_k\times P_k$, equipped with norms
	\begin{equation}\label{eq:xzero:norm}
	\begin{aligned}
 	       \|x_k\|_{X_{0,k}}^2  &:=\|(u_k,p_k)\|_{X_{0,k}}^2 := \|u_k\|_{U_{0,k}}^2 + \|p_k\|_{P_{0,k}}^2,\\
	      		\|u_k\|_{U_{0,k}}^2 &:=  h_k^{-2} \|u_k\|_{L^2(\Omega)}^2+\beta \|u_k\|_{L^2(\Omega)}^2\mbox{ and}\\
		       	\|p_k\|_{P_{0,k}}^2 &:=  \nnsup{w_k}{L^2(\Omega)} \frac{h_k^{-2} (p_k,w_k)_{L^2(\Omega)}^2}{h_k^{-2} \|w_k\|_{L^2(\Omega)}^2
				+\beta \|w_k\|_{L^2(\Omega)}^2}\\
				&=  h_k^{-2}  (h_k^{-2} +\beta)^{-1} \|p_k\|_{L^2(\Omega)}^2.
	\end{aligned}
	\end{equation}
	Note that this norm is obtained from the norm $\|\cdot\|_X$ by ``replacing all differentials by $h_k^{-1}$''. This is a common construction
	principle, cf.~\cite{Hackbusch:1985}.
	However, in~\cite{Olshanskii:2012} the norm
	\begin{equation}\label{eq:olsh:norm}
	\begin{aligned}
 	       \|x_k\|_{\tilde{X}_{0,k}}^2  &:=\|(u_k,p_k)\|_{\tilde{X}_{0,k}}^2 := \|u_k\|_{U_{0,k}}^2 + \|p_k\|_{P_{1,k}}^2,\\
	      		\|u_k\|_{U_{0,k}}^2 &:=  h_k^{-2} \|u_k\|_{L^2(\Omega)}^2+\beta \|u_k\|_{L^2(\Omega)}^2\mbox{ and}\\
       			\|p_k\|_{P_{1,k}}^2 &:=  \nnsup{w_k}{U_k} \frac{(p_k,\nabla \cdot
				w_k)_{L^2(\Omega)}^2}{\|w_k\|_{H^1(\Omega)}^2+\beta \|w_k\|_{L^2(\Omega)}^2}.
	\end{aligned}
	\end{equation}
	was used. Note that the norm for the velocity field coincides with our choice but
	the norm for the pressure distribution is still the original norm, as introduced in~\eqref{eq:x1:norm}.
\end{remark}

\section{Smoother and proof of the smoothing property}\label{sec:3aa}

The choice of an appropriate smoother is a key issue in constructing
a coupled multigrid method for an indefinite problem.
In this paper, we introduce two kinds of smoothers. The first smoother
is appropriate for a large class of problems including the model problem:
the \emph{normal equation smoother}, cf.~\cite{Brenner:1996}, which reads
as follows.
\begin{equation}\nonumber
        \ul{x}^{(0,m)}_k := \ul{x}^{(0,m-1)}_k + \tau
                         \underbrace{\mcL_k^{-1} \mcA_k \mcL_k^{-1}}_{\displaystyle
\hat{\mcA}_k^{-1}:=}
                        \left(\ul{\bff}_k -\mcA_k \;\ul{x}^{(0,m-1)}_k\right)
                \qquad \mbox{for } m=1,\ldots,\nu.
\end{equation}
Here, a fixed~$\tau>0$ has to be chosen such that the
spectral radius~$\rho(\tau \hat{\mcA}_k^{-1}\mcA_k)$ is bounded away
from~$2$ on all grid levels~$k$ and for all choices of the parameter~$\beta$.

\begin{remark}\label{rem:on:smoothers1}
	In Section~6.1 in~\cite{Olshanskii:2012} a normal equation smoother has been proposed. (There,
	this kind of smoothers was called \emph{distributive smoother}.) Note that this class of smoothers
	depends on the choice of the Hilbert space norm. In the present paper, a scaled $L^2$-norm is used. Therefore, the Riesz
	isomorphism $\mathcal{L}_k$ can be easily inverted. However, in~\cite{Olshanskii:2012}
	a different norm was chosen, cf. Remark~\ref{rem:on:norms1}. For that norm, the realization
	of the inverse of the Riesz isomorphism involves the solution of one Poisson problem (for the pressure variable). So,
	two such problems have to be solved (approximately) in each smoothing step, cf. Section~8 in~\cite{Olshanskii:2012}.
\end{remark}

Using a standard inverse inequality, one can show that
\begin{equation*}
	\tag{\textbf{A2}} \|x_k\|_{X} \les \hnorm x_k\hnorm_k 
\end{equation*}
is satisfied for all $x_k \in X_k$. Based on this result, using an eigenvalue 
analysis one can show the following lemma, cf.~\cite{Brenner:1996}.
\begin{lemma}
        The damping parameter $\tau>0$ can be chosen independently
        of the grid level~$k$ and the choice of the parameter~$\beta$ such that
        \begin{equation*}
                \tau\, \rho(\hat{\mcA}_k^{-1}\mcA_k) \le 2-\epsilon < 2,
        \end{equation*}
        holds for some constant $\epsilon>0$. For this choice of~$\tau$, there is a
        constant~$C_S>0$, independent of the grid level~$k$ and the choice of the
	parameter~$\beta$, such that the smoothing property~\eqref{eq:smp} 
	is satisfied with rate
        $
                \eta(\nu):=C_S \nu^{-1/2}.
        $
\end{lemma}

Certainly, the smoothing procedure~\eqref{eq:smp} should be efficient-to-apply.
Using the fact, that the mass matrices $M_{U,k}$ and $M_{P,k}$ 
in~\eqref{eq:defMcL} and their diagonals are
spectrally equivalent under weak assumptions, for the practical
realization of the normal equation smoother these mass matrices can be
replaced by their diagonals.

The second smoother, which we propose, is a \emph{Uzawa type smoother}, 
cf.~\cite{Schoeberl:Zulehner:2002}. Here, one step of the smoother to
compute $\ul{x}^{(0,m)}_k=(\ul{u}^{(0,m)}_k,\ul{p}^{(0,m)}_k)$ based 
on $\ul{x}^{(0,m-1)}_k=(\ul{u}^{(0,m-1)}_k,\ul{p}^{(0,m-1)}_k)$ reads as follows:
\begin{align*}
	\ul{u}^{(0,m-1/2)}_k & := \ul{u}^{(0,m-1)}_k
		+ \tau \hat{A}_k^{-1} \left( \ul{f}_{k} - A_k \ul{u}^{(0,m-1)}_k - B_k^T \ul{p}^{(0,m-1)}_k\right)\\
	\ul{p}^{(0,m)}_k & := \ul{p}^{(0,m-1)}_k
		- \sigma \hat{S}_k^{-1} \left( \ul{g}_{k} - B_k \ul{u}^{(0,m-1/2)}_k \right)\\
	\ul{u}^{(0,m)}_k & := \ul{u}^{(0,m-1)}_k
		+ \tau \hat{A}_k^{-1} \left( \ul{f}_{k} - A_k \ul{u}^{(0,m-1)}_k - B_k^T \ul{p}^{(0,m)}_k\right),
\end{align*}
where $\hat{A}_k$ and $\hat{S}_k$ are the (1,1)-block and the (2,2)-block of $\mcL_k$, respectively. The
damping parameters $\tau$ and $\sigma$ have to be chosen independently of the choice of $\beta$
and of the grid level $k$ such that~\eqref{eq:def:tau:sigma} holds. 
The smoother can be rewritten in the compact notation~\eqref{eq:sm:comp}, where
\begin{equation*}
	\hat{\mcA}_k := \left(
		\begin{array}{cc}
			\tau^{-1} \hat{A}_k & B_k^T \\
			B_k & \tau B_k \hat{A}_k^{-1} B^T - \sigma^{-1} \hat{S}_k
		\end{array}
	\right).
\end{equation*}
The smoothing property can be shown using the theory introduced in~\cite{Schoeberl:Zulehner:2002}.
\begin{lemma}
	Let~$\hat{A}_k$ and $\hat{S}_k$ be the (1,1)-block and the
		(2,2)-block of $\mcL_k$, respectively. Then $\tau>0$ and $\sigma>0$ 
	can be chosen independently of the grid level and the choice of the parameter
	$\beta$ such that
	\begin{equation}\label{eq:def:tau:sigma}
		\tau^{-1} \hat{A}_k \ge A_k
		\qquad \mbox{and} \qquad 
		\sigma^{-1} \hat{S}_k \ge \tau B_k \hat{A}_k^{-1} B_k^T.
	\end{equation}

	For this choice of~$\tau$ and~$\sigma$, there is a
        constant~$C_S>0$, independent of the grid level~$k$ and the choice of the
	parameter~$\beta$, such that the smoothing property~\eqref{eq:smp} 
	is satisfied with rate
        $
                \eta(\nu):=C_S \nu^{-1/2}.
        $
\end{lemma}
\begin{proof}
	The fact that $\tau>0$ and $\sigma>0$ can be chosen independently
	of the grid level and the choice of the parameter $\beta$ follows
	from standard inverse inequalities.

	For the smoothing property, we apply Theorem~4 in~\cite{Schoeberl:Zulehner:2002}
	with the choice $\mathcal{K}_k:= \mcL_k^{-1/2} \mcA_k \mcL_k^{-1/2}$ and
	$\hat{\mathcal{K}}_k:= \mcL_k^{-1/2} \hat{\mcA}_k \mcL_k^{-1/2}$. This immediately
	implies
	\begin{equation*}
		\|\mcL_k^{-1/2}\mcA_k (I-\hat{\mcA}_k\mcA_k)^{\nu}\mcL_k^{-1/2}\|_{\ell^2}
			\le \eta(\nu) \|\mcL_k^{-1/2}\mcA_k\mcL_k^{-1/2}\|_{\ell^2}.
	\end{equation*}
	As $\ul{x}_k^T\mcA_k\ul{\tilde{x}}_k =
		\mcB(x_k,\tilde{x}_k) \les \|x_k\|_X \|\tilde{x}_k\|_X
		\les \hnorm x_k\hnorm_{k} \hnorm\tilde{x}_k\hnorm_{k}
		= \|\ul{x}_k\|_{\mcL_k} \|\ul{\tilde{x}}_k\|_{\mcL_k}$
	holds for all $x_k, \tilde{x}_k\in X_k$, we obtain
	$\|\mcL_k^{-1/2}\mcA_k\mcL_k^{-1/2}\|_{\ell^2}\les 1$, which finishes the proof.
\end{proof}

\begin{remark}\label{rem:on:smoothers2}
	Uzawa type smoothers have also been considered in Section 6.2 in~\cite{Olshanskii:2012}. As
	for the case of normal equation smoothers (distributive smoothers) due to the particular norm
	that was chosen in~\cite{Olshanskii:2012}, again a discrete Poisson problem for the pressure variable has to
	be solved in each smoothing step, cf. Section~8 in~\cite{Olshanskii:2012}.
\end{remark}

\section{Proof of the approximation property}\label{sec:3a}

The proof of the approximation property~\eqref{eq:apprp} is done using the
approximation theorem introduced in~\cite{Takacs:Zulehner:2012}
which requires besides the conditions~\citecond{(A1)}
and~\citecond{(A1a)} two more conditions (conditions~\citecond{(A3)} and 
\citecond{(A4)}) involving, besides the Hilbert space $X$, two more Hilbert
spaces $X_{-,k}:=(X_-,{\|\cdot\|_{X_{-,k}}})$ and
 $X_{+,k}:=(X_+,{\|\cdot\|_{X_{+,k}}})$, which are chosen as follows.

As weaker space we choose $X_-:= U_-\times P_-$, where
$U_-:=[L^2(\Omega)]^2$ and $P_-:= [H^1_0(\Omega)\cap L^2_0(\Omega)]^*$.
These Hilbert spaces are equipped with norms
\begin{align*}
        \|x\|_{X_{-,k}}^2 & := \|(u,p)\|_{X_{-,k}}^2 := \|u\|_{U_{-,k}}^2 +\|p\|_{P_{-,k}}^2,\\
        \|u\|_{U_{-,k}}^2 &:=h_k^{-2} \|u\|_{L^2(\Omega)\cap \beta [H^1_0(\Omega)]^*}^2 \mbox{ and}\\
        \|p\|_{P_{-,k}}^2 &:=h_k^{-2} \|p\|_{[H^1_0(\Omega)]^* + \beta^{-1} L^2_0(\Omega)}^2.
\end{align*}

\begin{remark}
	The idea behind the construction of the norm $\|\cdot\|_{X_{-,k}}$ 
	is is to take the norm $\|\cdot\|_X$ and ``replace all occurrences
	of $H^1$ by $h_k^{-1} L^2$ and all occurrences of $L^2$ by $h_k^{-1} H^{-1}$''. 
	This is different to the idea of constructing the norm $\hnorm\cdot\hnorm_{k}$, where
	 ``only the $H^1$-terms are replaced by $h_k^{-1}L^2$ and the $L^2$-terms are kept unchanged''. 
\end{remark}

In the following, we show the approximation property in the norm $\|\cdot\|_{X_{-,k}}$,
i.e.,
\begin{equation}\label{eq:apprp:raute}
                       \| x_k^{(1)}-x_k^*\|_{X_{-,k}} \le
                       \tilde{C}_A \sup_{\tilde{x}_k\in X_k} 
                       \frac{\mcB\left(x_k^{(0,\nu)}-x_k^*,\tilde{x}_k\right)}{\|\tilde{x}_k\|_{X_{-,k}}}.
\end{equation}
We will show below that this version of the approximation property
is \emph{stronger} than the required estimate \eqref{eq:apprp}, cf. Lemma~\ref{lem:equiv}.
In classical approaches, cf.~\cite{Takacs:Zulehner:2012,Olshanskii:Peters:Reusken:2006},
the approximation property is directly shown in the norm $\hnorm\cdot\hnorm_{k}$.

Note that dual spaces are $(X_-)^*:= (U_-)^*\times (P_-)^*$, where
$(U_-)^*=[L^2(\Omega)]^2$ and $(P_-)^* = H^1_0(\Omega)\cap L^2_0(\Omega)$, equipped with norms
\begin{align*}
        \|\mcF \|_{(X_{-,k})^*}^2 & = \|(f,g)\|_{(X_{-,k})^*}^2 := \|f\|_{(U_{-,k})^*}^2 +\|g\|_{(P_{-,k})^*}^2,\\
        \|f\|_{(U_{-,k})^*}^2 &=h_k^{2} \|f\|_{L^2(\Omega)+\beta^{-1} H^1_0(\Omega)}^2  \mbox{ and}\\
        \|g\|_{(P_{-,k})^*}^2 &=h_k^{2} \|g\|_{H^1_0(\Omega)\cap \beta L^2_0(\Omega)}^2.
\end{align*}

As stronger space we choose
$X_{+}:=U_{+}\times P_{+}$, where
$U_{+}:=[H^2(\Omega)\cap H^1_0(\Omega)]^2$ and
$P_{+}:=H^1(\Omega)\cap L^2_0(\Omega)$, equipped with norms
\begin{align*}
        \|x\|_{X_{+,k}}^2 &:=\|(u,p)\|_{X_{+,k}}^2 := \|u\|_{U_{+,k}}^2+ \|p\|_{P_{+,k}}^2,\\
        \|u\|_{U_{+,k}}^2 &:= h_k^2 \|u\|_{H^2(\Omega)\cap \beta H^1(\Omega)}^2
	\mbox{ and}\\
        \|p\|_{P_{+,k}}^2 &:= h_k^2 \|p\|_{H^1(\Omega) + \beta^{-1} H^2(\Omega)}^2.
\end{align*}
For showing the approximation property, we need some auxiliary results.
The first result is a standard approximation error estimate:
\begin{theorem}\label{thrm:a3}
	On all grid levels $k$, the approximation error estimate
        \begin{equation}\tag{\textbf{A3}}
                \inf_{x_k\in X_k} \|x-x_k\|_X \les \| x \|_{X_{+,k}}\qquad \mbox{for all } x\in X_+
        \end{equation}
        is satisfied.
\end{theorem}

The proof of this theorem is rather easy, as it can be shown for $U$ and $P$ separately. For completeness, we give
a proof of this theorem in the Appendix.
The next step is to show a regularity result for the generalized
Stokes problem:
\begin{description}[style=sameline,leftmargin=1.1cm]
        \item[\citecond{(A4)}] For all grid levels $k$ and all $\mcF\in (X_-)^*$, the solution of the problem, 
        \begin{equation}\label{a4:problem}
               \mbox{find $x_{\mcF}\in X$ such that} \qquad \mcB(x_{\mcF},\tilde{x}) = \mcF(\tilde{x}) \qquad\mbox{ for all } \tilde{x}\in X,
        \end{equation}
        satisfies $x_{\mcF}\in X_+$ and $\|x_{\mcF}\|_{X_{+,k}} \les \|\mcF\|_{(X_{-,k})^*}$.
\end{description}
For showing~\citecond{(A4)}, we need a standard regularity result for the Poisson problem.
\begin{description}[style=sameline,leftmargin=1.1cm]
	        \item[\citecond{(R1)}] \emph{Regularity of the Poisson problem.}
	        For all $g\in L^2(\Omega)$, the solution of the problem, find $p_g\in H^1(\Omega)\cap L^2_0(\Omega)$ such that
	        \begin{equation*}
	          (\nabla p_g,\nabla \tilde{p})_{H^1(\Omega)} = (g,\tilde{p})_{L^2(\Omega)} \qquad\mbox{ for all }\tilde{p} \in H^1(\Omega)\cap L^2_0(\Omega),
	        \end{equation*}
	        satisfies $p_g\in H^2(\Omega)$ and
 	       $\|p_g\|_{H^2(\Omega)} \les \|g\|_{L^2(\Omega)}$.
\end{description}
\begin{lemma}
	The regularity assumption~\citecond{(R1)} holds for $\Omega\subset \mathbb{R}^2$ being a convex polygonal domain.
\end{lemma}

For a proof oft this lemma, see, e.g.,~\cite{Dauge:1988}.
The first step of the proof of~\citecond{(A4)}, is to show that~$x_{\mcF}$ is in the desired space.
\begin{lemma}\label{lem:step2}
	Suppose that $\Omega$ is such that the regularity assumption~\citecond{(R)} holds.
	Let $\mcF \in (X_-)^*$ and let $x_{\mcF}$ be the solution of~\eqref{a4:problem}. Then $x_{\mcF}\in X_{+}$.
\end{lemma}
\begin{proof}
	Let $\mcF(\tilde{u},\tilde{p}):= (f,\tilde{u})_{L^2(\Omega)}+(g,\tilde{p})_{L^2(\Omega)}$
	for $f \in [L^2(\Omega)]^2$ and $g \in H^1_0(\Omega)\cap L^2_0(\Omega)$.
	Rewrite the problem as follows:
        \begin{equation*}
          \begin{array}{lclclcl}
              (\nabla u_{\mcF},\nabla \tilde{u})_{L^2(\Omega)} &+&
              (p_{\mcF},\nabla \cdot \tilde{u})_{L^2(\Omega)} &=&
              (f,\tilde{u})_{L^2(\Omega)}- \beta
              (u_{\mcF},\tilde{u})_{L^2(\Omega)}\\
              (\nabla \cdot u_{\mcF},\tilde{p})_{L^2(\Omega)} 
              &&&=& (g,\tilde{p})_{L^2(\Omega)}
           \end{array}                
        \end{equation*}
	for all $\tilde{u} \in U$ and $\tilde{p}\in P$.
	As $f - \beta u_{\mcF} \in [L^2(\Omega)]^2$
	and $g\in H^1_0(\Omega)\cap L^2_0(\Omega)$,
	we obtain using regularity assumption~\citecond{(R)} 
	immediately that $x_{\mcF}\in X_{+}$.
\end{proof}

Note that the combination of the argument used in the proof of 
this lemma and condition~\citecond{(A1)}
implies $\|x_{\mcF}\|_{X_{+,k}} \le C(\beta) \|\mcF\|_{(X_{-,k})^*}$, where
$C(\beta)$ is a constant depending on $\beta$. For showing a robust estimate,
we need to do some more work.

\begin{lemma}\label{lem:aux1}
	Suppose that $\Omega$ is such that the assumptions~\citecond{(R)} and~\citecond{(R1)} are satisfied.
	Let $\mcF \in (X_-)^*$ and let
	$x_{\mcF}=(u_{\mcF},p_{\mcF})$ be the solution of~\eqref{a4:problem}. Then
	\begin{equation*}
		\|p_{\mcF}\|_{P_{+,k}}^2 \les \|u_{\mcF}\|_{U_{+,k}}^2
			+ \|\mcF\|_{(X_{-,k})^*}^2.
	\end{equation*}
\end{lemma}

\begin{lemma}\label{lem:aux2}
	Suppose that $\Omega$ is such that the assumption~\citecond{(R)} is satisfied.
	Let $\mcF\in (X_-)^*$ and let
	$x_{\mcF}=(u_{\mcF},p_{\mcF})$ be the solution of~\eqref{a4:problem}. Then
	\begin{equation*}
			\|u_{\mcF}\|_{U_{+,k}}^2
					\les \|\mcF\|_{(X_{-,k})^*} \|x_{\mcF}\|_{X_{+,k}}.
	\end{equation*}
\end{lemma}

The proofs of these two lemmas can be found in the Appendix. Using them, we can show~\citecond{(A4)}:
\begin{theorem}\label{thrm:a4}
        If $\Omega$ is such that the regularity assumptions~\citecond{(R)} and~\citecond{(R1)}
        are satisfied, then the regularity statement~\citecond{(A4)} holds.
\end{theorem}
\begin{proof}
	Note that Lemma~\ref{lem:step2} already states 
	that $x_{\mcF}\in X_{+}$.
	It remains to show
	$
		\|x_{\mcF}\|_{X_{+,k}} \les \|\mcF\|_{(X_{-,k})^*}
	$
	Note that Lemma~\ref{lem:aux1} implies immediately
	\begin{align}\nonumber
			\|x_{\mcF}\|_{X_{+,k}}^2 = \|u_{\mcF}\|_{U_{+,k}}^2+\|p_{\mcF}\|_{P_{+,k}}^2  \les 
				\|u_{\mcF}\|_{U_{+,k}}^2
				+\|\mcF\|_{(X_{-,k})^*}^2
	\end{align}
	If we combine this result with the statement of
	Lemma~\ref{lem:aux2}, we obtain
	\begin{align}\nonumber
			\|x_{\mcF}\|_{X_{+,k}}^2  \le C \left(
				\|x_{\mcF}\|_{X_{+,k}}\|\mcF\|_{(X_{-,k})^*}
				+\|\mcF\|_{(X_{-,k})^*}^2 \right)
	\end{align}
	for some constant $C>0$ (independent of $k$ and $\beta$) which implies
	\begin{align}\nonumber
		\|x_{\mcF}\|_{X_{+,k}}  \le \frac{1}{2}\left(C + \sqrt{4C+C^2}\right)\|\mcF\|_{(X_{-,k})^*}, 
	\end{align}
	and the desired estimate. This finishes the proof.
\end{proof}

As we have shown that the conditions~\citecond{(A1)},
\citecond{(A1a)}, \citecond{(A3)} and~\citecond{(A4)} are satisfied, we can apply
Theorem~4.1 in~\cite{Takacs:Zulehner:2012} and obtain:
\begin{theorem}\label{thrm:main}
	Assume that $\Omega$ is such that \citecond{(R)} and \citecond{(R1)}
	are satisfied and assume that the problem is discretized using the Taylor-Hood element.
        Then the coarse-grid correction~\eqref{eq:method:cga} satisfies the approximation 
	property~\eqref{eq:apprp:raute}
	 where $\tilde{C}_A$ only depends on the constants that appear in the conditions~\citecond{(A1)},
	\citecond{(A1a)}, \citecond{(A3)} and~\citecond{(A4)}.
\end{theorem}

This shows \eqref{eq:apprp:raute}, the approximation property in the norm~$\|\cdot\|_{X_{-,k}}$, however we need~\eqref{eq:apprp}, the
approximation property in the norm~$\hnorm \cdot\hnorm_k$. The final step is to relate these two norms:
\begin{lemma}\label{lem:equiv}
	The estimate
	\begin{equation}\label{eq:equiv}
			\hnorm x_k\hnorm_k \les \|x_k\|_{X_{-,k}}
	\end{equation}
	 holds for all $x_k \in X_k$
\end{lemma}
\begin{proof}
	First note that it suffices to show $\hnorm u_k\hnorm_{U,k} \les \|u_k\|_{U_{-,k}}$ and $\hnorm p_k\hnorm_{P,k} \les \|p_k\|_{P_{-,k}}$.
	We obtain using the definition of $\|\cdot\|_{U_{-,k}}$, the choice $\tilde{u}:=u_k \in [H^1_0(\Omega)]^2$ and a standard inverse inequality
	\begin{align*}
		&\|u_k\|_{U_{-,k}}^2  =  \nnsup{\tilde{u}}{[L^2(\Omega)]^2} \frac{(u_k,\tilde{u})_{L^2(\Omega)}^2 } {h_k^{2}\|\tilde{u}\|_{L^2(\Omega) + \beta^{-1} H^1_0(\Omega)}^2 }
				 \ges \frac{(u_k,u_k)_{L^2(\Omega)}^2 } {h_k^{2}\|u_k\|_{L^2(\Omega) + \beta^{-1} H^1_0(\Omega)}^2 }\\
				& = \frac{(u_k,u_k)_{L^2(\Omega)}^2 } {\inf_{w\in [H^1_0(\Omega)]^2} h_k^{2}\|w- u_k\|_{L^2(\Omega)}^2 + h_k^2 \|w\|_{\beta^{-1} H^1(\Omega)}^2 }\\
				& \ges \frac{(u_k,u_k)_{L^2(\Omega)}^2 } {\inf_{w_k\in U_k} h_k^{2}\|w_k- u_k\|_{L^2(\Omega)}^2 + h_k^2 \| w_k\|_{\beta^{-1} H^1(\Omega)}^2 }\\
				& \ges \frac{ (u_k,u_k)_{L^2(\Omega)}^2 } {\inf_{w_k\in U_k}  h_k^{2}\|w_k-u_k\|_{L^2(\Omega)}^2 + \|w_k\|_{\beta^{-1} L^2(\Omega)}^2 }
				 = (\beta+h_k^{-2}) \|u_k\|_{L^2(\Omega)}^2  = \hnorm u_k\hnorm_{U,k}^2.
	\end{align*}	
	
	For the proof of the inequality for the pressure variable,
	we define on each grid level a function $\psi_k: \Omega \rightarrow \mathbb{R}$ as follows. The
	function $\psi_k$ is $0$ on all vertices which are located on $\partial\Omega$,
	$1$ on all other vertices (interior vertices) and linear within each element (Courant element).

	In what follows, $\hat{P}_k$ is the space of functions obtained
	by discretizing the space $H^1(\Omega)$ using the standard Courant element. Note that -- as we have
	used the Taylor Hood element -- the identity $\hat{P}_k = \{p_k+a\;:\;p_k\in P_k\mbox{ and } a \in \mathbb{R}\}$ holds.	
	
	Note that we obtain using the definition of $\|\cdot\|_{P_{-,k}}$ and $q \subseteq L^2_0(\Omega)$ that
	\begin{align*}
		\|p_k\|_{P_{-,k}}^2 & = \hspace{-.1cm}  \nnsup{q}{H^1_0(\Omega)\cap L^2_0(\Omega)} \frac{(p_k,q)_{L^2(\Omega)}^2 } {h_k^{2}\|q\|_{H^1_0(\Omega)\cap \beta L^2(\Omega)}^2 }
			 = \hspace{-.1cm} \nnsup{q}{H^1_0(\Omega)\cap L^2_0(\Omega)} \frac{(p_k-a,q)_{L^2(\Omega)}^2 } {h_k^{2}\|q\|_{H^1_0(\Omega)\cap \beta L^2(\Omega)}^2 }
	\end{align*}
	is satisfied, where $a\in \mathbb{R}$ is chosen such that $ (p_k-a)\psi_k \in L^2_0(\Omega)$.
	Observe that $q:=(p_k-a)\psi_k \in H^1_0(\Omega)\cap L^2_0(\Omega)$. Using this choice, we obtain
	\begin{align*}
		\|p_k\|_{P_{-,k}}^2 & \ge \frac{(p_k-a,(p_k-a)\psi_k )_{L^2(\Omega)}^2 } {h_k^2\|(p_k-a)\psi_k\|^2_{H^1(\Omega)} + \beta h_k^2\|(p_k-a)\psi_k\|^2_{L^2(\Omega)} }.
	\end{align*}
	Using~$|\psi_k|\le 1$, we obtain $\|(p_k-a)\psi_k\|^2_{L^2(\Omega)}\le \|p_k-a\|^2_{L^2(\Omega)}$. The following two inequalities
	hold:
	\begin{align}\label{lem:equiv:aux2}
			\|\psi_k p_k\|_{H^1(\Omega)}^2 &\les h_k^{-2} \|p_k\|_{L^2(\Omega)}^2 \mbox{ and}\\
	 			\label{eq:kinv}
			(p_k,\psi_k p_k)_{L^2(\Omega)} &\ges (p_k,p_k)_{L^2(\Omega)}.
	\end{align}
	Proofs are given in the Appendix.
	Using these inequalities and $p_k\in L^2_0(\Omega)$, we obtain
	\begin{align*}
		\|p_k\|_{P_{-,k}}^2 & \ges \frac{(p_k-a,p_k-a)_{L^2(\Omega)}^2 } {\|p_k-a\|^2_{L^2(\Omega)} + \beta h_k^2\|p_k-a\|^2_{L^2(\Omega)} }\\
					& = (1 + \beta h_k^{2})^{-1} \|p_k-a\|_{L^2(\Omega)}^2
					 \ge (1 + \beta h_k^{2})^{-1} \|p_k\|_{L^2(\Omega)}^2  = \hnorm p_k\hnorm_{P,k}^2,
	\end{align*}
	which finishes the proof.	
\end{proof}

Using this lemma and the approximation property~\eqref{eq:apprp:raute}, we obtain 
the approximation property~\eqref{eq:apprp}. Together with the smoothing property, we obtain:
\begin{theorem}
        Assume that
        \begin{itemize}
                \item $\Omega$ is such that the regularity assumptions~\citecond{(R)} and~\citecond{(R1)} are satisfied,
                \item the problem is discretized using the Taylor-Hood element and
                \item one of the smoothers proposed in this paper is used.
        \end{itemize}
        Then
        \begin{equation*}
                \hnorm x_k^{(1)}-x_k^*\hnorm_k \le q(\nu) \hnorm x_k^{(0)}-x_k^*\hnorm_k
        \end{equation*}
        holds
        with $q(\nu):=C_S\, C_A\, \nu^{-1/2}$, where the constants~$C_A$ and~$C_S$ are independent of the grid
				level~$k$ and the choice of the parameter~$\beta$. For $\nu > (C_S\, C_A)^2$, the two-grid method converges 
				with convergence rate $q(\nu)<1$.
\end{theorem}

The convergence of the W-cycle multigrid follows under weak assumptions, cf.~\cite{Hackbusch:1985}.

\section{Numerical Results}\label{sec:4}

In this section, we illustrate the convergence theory presented within
this paper with numerical results.
For the numerical experiments, the domain~$\Omega$ was chosen to be the unit square $\Omega:=(0,1)^2$.
As mentioned in Section~\ref{sec:2}, the weak inf-sup-condition~\eqref{eq:cond:s} 
can be shown for the Taylor-Hood element only if at least one vertex of each element
is in the interior of the domain~$\Omega$. As this is not satisfied for the the standard 
decomposition of the unit square into two triangular elements, we choose the coarsest 
grid level~$k=0$ to be a decomposition of the domain~$\Omega$ into $8$ triangles, 
as seen in Figure~\ref{fig:1}. The grid levels~$k=1,2,\ldots$ were constructed
by uniform refinement, i.e., every triangle was decomposed into four
subtriangles.
\begin{figure}[ht]%
\begin{center}
        \includegraphics[scale=.4]{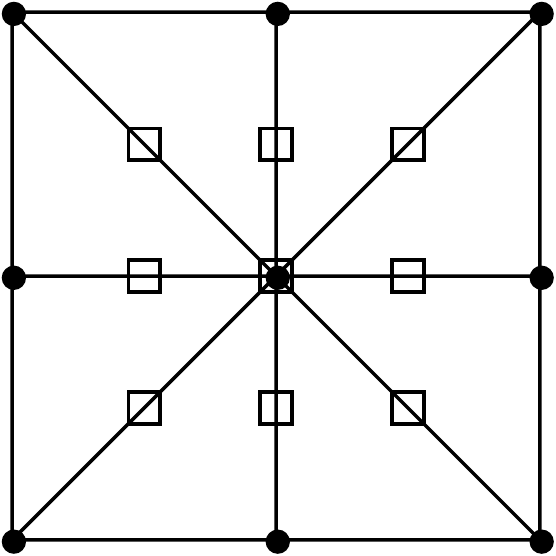}\qquad
        \includegraphics[scale=.4]{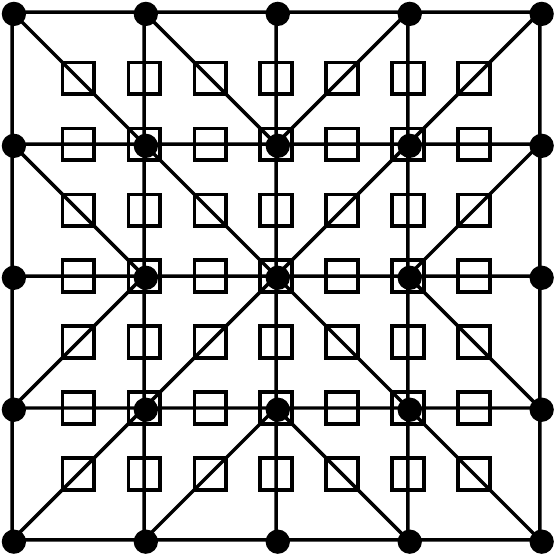}
        \caption{Discretization on grid levels $k=1$ and $k=2$, where the squares denote the degrees of freedom of (the components of) 
	the velocity field $u$ and the dots denote the degrees of freedom of pressure distribution $p$}
        \label{fig:1}        
\end{center}
\end{figure}%

The right-hand-side functions $f$ and $g$ have been chosen such that
the solution of the problem on grid level $k$ is the $L^2$-projection of the exact
solution 
\begin{equation*}
	u(\xi_1,\xi_2) := \phi(\xi_1,\xi_2) \left(\begin{array}{c}\xi_2-\tfrac12\\\tfrac12-\xi_1\end{array}\right)\qquad\mbox{and}\qquad p(\xi_1,\xi_2):=\phi(\xi_1,\xi_2),
\end{equation*} into $X_k$, where
	$\phi(\xi_1,\xi_2) := \max\left\{0,\min\left\{1,2-4\sqrt{(\xi_1-\tfrac12)^2+(\xi_2-\tfrac12)^2}\right\}\right\}$.

For solving the discretized problem, we have used the proposed W-cycle multigrid method. As mentioned
above, the matrix $\mcL_k$, introduced in Section~\ref{sec:3} can be replaced by any spectrally equivalent
matrix. As $\mcL_k$ should be easy-to-invert, $\mcL_k$ should be a diagonal matrix. The convergence theory indicates
how~$\mcL_k$ should be scaled, however only up to constants. Numerical experiments indicated that it is
reasonable to choose 
\begin{equation}\label{eq:numintroduc}
        \mcL_k :=
        \left(
                \begin{array}{cc}
                        \hat{A}_k \\ & \hat{S}_k
                \end{array}
        \right), \quad\mbox{where}\quad \hat{A}_k:=\mbox{diag } A_k \quad\mbox{and}\quad \hat{S}_k:= \mbox{diag } B_k\hat{A}_k^{-1}B_k^T,
\end{equation}
i.e., to choose these constants as they appear naturally in the diagonal of~$A_k$, the (1,1)-block of $\mathcal{A}_k$, and in the
diagonal of the inexact Schur complement. As this choice of $\mcL_k$ is spectrally equivalent to the matrix
$\mcL_k$, introduced in Section~\ref{sec:3}, this choice is still covered by the convergence theory.
For this choice, the damping parameter was chosen to be~$\tau=0.35$, for all grid levels~$k$ and all choices of~$\beta$.

For the Uzawa type smoother, the matrices $\hat{A}_k$ and
$\hat{S}_k$ have been chosen as introduced in~\eqref{eq:numintroduc} and the damping parameters have been chosen
to be $\tau= \sigma = 0.8$ for all grid levels~$k$ and all choices of~$\beta$. Again, the choice of this section
is covered by the convergence theory.

The number of iterations and the convergence rate were measured as
follows: we start with $\ul{x}_k^{(0)}$ and measure the reduction of the error in each step
using the norm $\hnorm \cdot \hnorm_k$. The iteration was stopped when
the initial error was reduced by a factor of $\epsilon = 10^{-9}$. The convergence rate
$q$ is the mean convergence rate in this iteration, i.e.,
\begin{equation*}
        q = \left(\frac{\hnorm x_k^{(n)}-x_k^* \hnorm_k}{\hnorm x_k^{(0)}-x_k^*\hnorm_k}\right)^{1/n},
\end{equation*}
where $n$ is the number of iterations needed to reach the stopping
criterion. Here, $x_k^*$ is the exact solution and $x_k^{(i)}$ is the $i$-th iterate.


\begin{table}[ht]%
\begin{center}
  \begin{tabular}{p{0.30cm}p{1.0cm}p{0.30cm}p{1.0cm}p{0.30cm}p{1.0cm}p{0.30cm}p{1.0cm}p{0.30cm}p{1.0cm}p{0.30cm}p{1.0cm}}
  \hline\noalign{\smallskip}
     \multicolumn{2}{l}{$\nu=1+1$} &
		 \multicolumn{2}{l}{$\nu=2+2$} &
		 \multicolumn{2}{l}{$\nu=3+3$} &
	   \multicolumn{2}{l}{$\nu=4+4$} &
     \multicolumn{2}{l}{$\nu=8+8$} &
     \multicolumn{2}{l}{$\nu=16+16$} \\  
  \hline\noalign{\smallskip}
       $n$ & $q$& $n$ & $q$& $n$ & $q$& $n$ & $q$& $n$ & $q$& $n$ & $q$\\
        \noalign{\smallskip}\hline\noalign{\smallskip}
		\multicolumn{12}{l}{Normal equation smoother}\\
        \noalign{\smallskip}\hline\noalign{\smallskip}
	         $88$&$0.789$&$46$&$0.634$&$30$&$0.496$&$24$&$0.412$&$17$&$0.293$&$11$&$0.148$\\
        \noalign{\smallskip}\hline\noalign{\smallskip}
		\multicolumn{12}{l}{Uzawa type smoother}\\
        \noalign{\smallskip}\hline\noalign{\smallskip}
	       \multicolumn{2}{l}{divergent} & \multicolumn{2}{l}{divergent} &$14$&$0.212$&$20$&$0.347$&$ 6$&$0.031$&$ 5$&$0.008$\\
          \noalign{\smallskip}\hline\noalign{\smallskip}
    \multicolumn{12}{l}{Asymptotic decay of the convergence rate predicted by theory ($\nu^{-1/2}$)}\\
    \noalign{\smallskip}\hline\noalign{\smallskip}
     &$0.707$&&$0.500$&&$0.408$&&$0.353$&&$0.250$&&$0.177$\\  
     \noalign{\smallskip}\hline\noalign{\smallskip}
        \end{tabular}
        \caption{Number of iterations $n$ and mean convergence rate $q$ for the normal equation smoother
	and the Uzawa type smoother depending with $\nu=\nu_{pre}+\nu_{post}$ smoothing steps on grid level $k=4$ for $\beta=1$}
        \label{tab:0}        
\end{center}
\end{table}%

\begin{table}[ht]%
\begin{center}
        \begin{tabular}{lp{0.26cm}p{0.85cm}p{0.26cm}p{0.85cm}p{0.26cm}p{0.85cm}p{0.26cm}p{0.85cm}p{0.26cm}p{0.85cm}p{0.26cm}p{0.85cm}}
        \hline\noalign{\smallskip}
               & \multicolumn{2}{l}{$\beta=0$} &
		             \multicolumn{2}{l}{$\beta=10^{2}$} &
                 \multicolumn{2}{l}{$\beta=10^{4}$} &
                 \multicolumn{2}{l}{$\beta=10^{6}$} &
                 \multicolumn{2}{l}{$\beta=10^{8}$} &
                 \multicolumn{2}{l}{$\beta=10^{10}$} \\
               & $n$ & $q$& $n$ & $q$& $n$ & $q$& $n$ & $q$& $n$ & $q$& $n$ & $q$\\
        \noalign{\smallskip}\hline\noalign{\smallskip}
		$k=4$&30&0.496&30&0.493&30&0.496&68&0.736&71&0.745&71&0.745\\
		$k=5$&29&0.489&29&0.488&22&0.388&64&0.722&70&0.744&71&0.745\\
		$k=6$&29&0.484&29&0.486&27&0.457&52&0.670&70&0.743&71&0.745\\
		$k=7$&28&0.475&28&0.475&28&0.470&35&0.553&70&0.742&71&0.745\\
		$k=8$&28&0.469&28&0.469&28&0.468&20&0.347&67&0.732&71&0.746\\
        \noalign{\smallskip}\hline\noalign{\smallskip}
        \end{tabular}
        \caption{Number of iterations $n$ and mean convergence rate $q$ for the normal equation smoother with $\nu=3+3$ smoothing steps}
        \label{tab:1}        
\end{center}
\end{table}%

\begin{table}[ht]%
\begin{center}
        \begin{tabular}{lp{0.26cm}p{0.85cm}p{0.26cm}p{0.85cm}p{0.26cm}p{0.85cm}p{0.26cm}p{0.85cm}p{0.26cm}p{0.85cm}p{0.26cm}p{0.85cm}}
        \hline\noalign{\smallskip}
               & \multicolumn{2}{l}{$\beta=0$} &
		             \multicolumn{2}{l}{$\beta=10^{2}$} &
                 \multicolumn{2}{l}{$\beta=10^{4}$} &
                 \multicolumn{2}{l}{$\beta=10^{6}$} &
                 \multicolumn{2}{l}{$\beta=10^{8}$} &
                 \multicolumn{2}{l}{$\beta=10^{10}$} \\
               & $n$ & $q$& $n$ & $q$& $n$ & $q$& $n$ & $q$& $n$ & $q$& $n$ & $q$\\
        \noalign{\smallskip}\hline\noalign{\smallskip}
		$k=4$&14&0.212&13&0.201& 6&0.030& 7&0.051& 8&0.059& 8&0.059\\
		$k=5$&13&0.194&13&0.193& 9&0.095& 7&0.043& 7&0.050& 7&0.050\\
		$k=6$&12&0.176&12&0.176&11&0.145& 6&0.024& 7&0.047& 7&0.048\\
		$k=7$&12&0.166&12&0.166&11&0.150& 5&0.013& 7&0.044& 7&0.045\\
		$k=8$&11&0.147&11&0.147&11&0.145& 8&0.058& 7&0.038& 7&0.043\\
        \noalign{\smallskip}\hline\noalign{\smallskip}
        \end{tabular}
        \caption{Number of iterations $n$ and mean convergence rate $q$ for the Uzawa type smoother with $\nu=3+3$ smoothing steps}
        \label{tab:2}        
\end{center}
\end{table}%

In Table~\ref{tab:0} we compare for a fixed grid level (level~$k=4$) and the fixed choice
$\beta=1$ the convergence rates for several choices of $\nu$, the number of pre- and
post-smoothing steps. We see that the convergence rate behaves approximately like~$\nu^{-1/2}$,
if the number of smoothing steps is increased. This is consistent with the theory which
guarantees the convergence rate being bounded by $C\,\nu^{-1/2}$. We observe that the preconditioned
normal equation smoother already converges for $\nu=1+1$ smoothing steps, while
for the Uzawa type smoother $\nu=3+3$ smoothing steps are necessary.

In Tables~\ref{tab:1} and \ref{tab:2} we compare various grid levels~$k$ and choices of the parameter~$\beta$.
We have used a fixed choice of $\nu=3+3$ smoothing 
steps. First we observe that, for both smoothers, the number of iterations seems to be well-bounded 
for all grid levels~$k$ which yields an optimal convergence behavior. Moreover, we see that the number 
of iterations is also well-bounded for a wide range of choices of the parameter~$\beta$, i.e., we observe
also robust convergence as predicted by the convergence theory.
Comparing both kinds of smoothers, we see that the Uzawa type smoother leads to much faster convergence
rates than the preconditioned normal equation smoother. Note that moreover the computational complexity of the
Uzawa type smoother (per iteration) is slightly smaller than the complexity of the normal equation smoother.

It has to be mentioned that for the model problem, also the V-cycle multigrid method
converges with rates comparable to the convergence rates of the W-cycle multigrid method. However,
the V-cycle is not covered by the convergence theory.

The numerical experiments done by the author have shown that the convergence rates can be
improved slightly by adjusting the choice of the parameters to the grid levels and the choice of~$\beta$.
However, the main goal of this paper is to show that the proposed method also works well for
fixed choices of the parameter.

\section{Conclusions and Further Work}\label{sec:5}

In the present paper we have proposed a coupled multigrid solver for the generalized
Stokes problem where the smoothing property is shown in the scaled $L^2$-norm 
${\hnorm\cdot\hnorm_k}$. This allows to construct a multigrid method where the smoother
is a simple linear iteration (which consists only of divisions and the multiplication of vectors
with the system matrix~$\mcA_k$). In the present paper, a preconditioned 
normal equation smoother and an Uzawa type smoother have been chosen but it 
seems possible to find also other smoothers
which satisfy the smoothing property in the norm~${\hnorm\cdot\hnorm_k}$.

The convergence rates observed for the multigrid method proposed in the present paper are comparable with
the rates observed for the methods proposed in~\cite{Olshanskii:2012}. Note that for applying the methods 
proposed in the named paper, it is necessary to solve a Poisson problem in each smoothing step. This is not
needed for the method proposed in the present paper. The main contribution of this paper is a new way
of setting up the norms where the convergence is sown in. The technique of the convergence proof that
has been applied in the present paper is also extendable to the Stokes control problem, cf.~\cite{Takacs:2013a}. 

\textbf{Acknowledgments.} The author thanks Markus Kollmann for providing code used to compute 
the numerical results presented in this paper. Moreover, the support of the Austrian Science Funds (FWF), the
University of Oxford and the Technische Universit\"at Chemnitz are gratefully acknowledged. The author
thanks the referee for many helpful suggestions.

\section{Appendix}
\def\proof{\emph{Proof of Theorem~\ref{thrm:a3}.}}
\begin{proof}
  Note that it suffices to show approximation error results for the 
	individual variables separately. Using a standard interpolation operator 
	$\Pi_k:[H^2(\Omega)]^2\rightarrow U_k$, we obtain
  \begin{equation*}
		\|u - \Pi_k u\|_{L^2(\Omega)}^2 \les h_k^2 \|u \|_{H^1(\Omega)}^2
		\quad\mbox{and}\quad
		\|u - \Pi_k u\|_{H^1(\Omega)}^2 \les h_k^2 \|u \|_{H^2(\Omega)}^2
  \end{equation*}
	for all $u\in [H^2(\Omega)]^2$ and therefore
  \begin{align*}
		& \inf_{u_k\in U_k} \|u-u_k\|_{U}^2 \le \|u-\Pi_ku\|_{U}^2
                =\|u-\Pi_ku\|_{H^1(\Omega)}^2+\beta \|u-\Pi_ku\|_{L^2(\Omega)}^2
		\\&\qquad \les h_k^2 \left(\|u\|_{H^2(\Omega)}^2
		+\beta \|u\|_{H^1(\Omega)}^2\right)
                =  \|u\|_{U_{+,k}}^2.
  \end{align*}
	A similar argument is possible for the pressure distribution. The standard approximation error estimates
  \begin{equation*}
		\inf_{p_k\in P_k} \|p - p_k\|_{L^2(\Omega)}^2 \les h_k^2 \|p\|_{H^1(\Omega)}^2
		\quad \mbox{and} \quad
		\inf_{p_k\in P_k} \|p - p_k\|_{H^1(\Omega)}^2 \les h_k^2 \|p\|_{H^2(\Omega)}^2
  \end{equation*}
	hold for all $p\in H^1(\Omega)$ or $p\in H^2(\Omega)$, respectively. They imply
  \begin{align*}
                &\inf_{p_k\in P_k} \|p-p_k\|_P^2
			=\hspace{-.5em}\inf_{\substack{p_1\in L^2(\Omega)\\p_2\in H^1(\Omega)\\p_1+p_2=p}}
			\inf_{p_{1,k}\in P_k} \|p_1-p_{1,k}\|_{ L^2(\Omega)}^2
			+\inf_{p_{2,k}\in P_k} \|p_2-p_{2,k}\|_{\beta^{-1} H^1(\Omega)}^2\\
			&\quad\le
			\inf_{\substack{p_1\in H^1(\Omega)\\p_2\in H^2(\Omega)\\p_1+p_2=p}}
			\inf_{p_{1,k}\in P_k} \|p_1-p_{1,k}\|_{ L^2(\Omega)}^2
			+\inf_{p_{2,k}\in P_k} \|p_2-p_{2,k}\|_{\beta^{-1} H^1(\Omega)}^2\\
			&\quad\les h_k^2
			\inf_{\substack{p_1\in H^1(\Omega)\\p_2\in H^2(\Omega)\\p_1+p_2=p}}
			\|p_1\|_{ H^1(\Omega)}^2
			+ \|p_2\|_{\beta^{-1} H^2(\Omega)}^2
		= h_k^2 \|p\|_{ H^1(\Omega)+\beta^{-1} H^2(\Omega)}^2,
        \end{align*}
	for all $p\in H^1(\Omega) \cap L^2_0(\Omega)$, which finishes the proof.
\end{proof}

\def\proof{\emph{Proof of Lemma~\ref{lem:aux1}.}}
\begin{proof}
	Let $\mcF(\tilde{u},\tilde{p}):= (f,\tilde{u})_{L^2(\Omega)}+(g,\tilde{p})_{L^2(\Omega)}$
	for $f \in [L^2(\Omega)]^2$ and $g \in H^1_0(\Omega)\cap L^2_0(\Omega)$.
	Choose $f_2 \in [H^1_0(\Omega)]^2$ arbitrarily and define $f_1:=f-f_2\in [L^2(\Omega)]^2$.
	Observe that
	\begin{equation*}
		(p_{\mcF},\nabla\cdot \tilde{u})_{L^2(\Omega)}
		= - (\nabla u_{\mcF},\nabla \tilde{u})_{L^2(\Omega)} 
		- \beta (u_{\mcF},\tilde{u})_{L^2(\Omega)} 
		+ (f_1+f_2,\tilde{u})_{L^2(\Omega)}
	\end{equation*}
	is satisfied for all $\tilde{u} \in [H^1_0(\Omega)]^2$. This is equivalent to
	\begin{equation*}
		-(\nabla p_{\mcF}, \tilde{u})_{L^2(\Omega)} 
		=  (\Delta u_{\mcF}, \tilde{u})_{L^2(\Omega)} 
		- \beta (u_{\mcF},\tilde{u})_{L^2(\Omega)} 
		+ (f_1+f_2,\tilde{u})_{L^2(\Omega)}
	\end{equation*}
	for all $\tilde{u} \in [H^1_0(\Omega)]^2$. Note that 
	$\nabla p_{\mcF}$,
	$\Delta u_{\mcF}$, $u_{\mcF}$, $f_1$ and $f_2$ are in 
	$[L^2(\Omega)]^2$. Therefore, the above statement holds for all 
	$\tilde{u} \in [L^2(\Omega)]^2$ (because
	$[H^1_0(\Omega)]^2$ is dense in $[L^2(\Omega)]^2$), particularly
	for $\tilde{u}:=\nabla \tilde{p}$,
	where $\tilde{p} \in H^1(\Omega) \cap L^2_0(\Omega)$. So, we obtain
	\begin{equation}\label{eq:lemaux0}
		-(\nabla p_{\mcF}, \nabla \tilde{p})_{L^2(\Omega)} 
		=  (\Delta u_{\mcF}, \nabla \tilde{p})_{L^2(\Omega)} 
		- \beta (u_{\mcF},\nabla \tilde{p})_{L^2(\Omega)} 
		+ (f_1+f_2,\nabla \tilde{p})_{L^2(\Omega)}
	\end{equation}
	for all $\tilde{p} \in  H^1(\Omega) \cap L^2_0(\Omega)$.

	Let $p_1 \in H^1(\Omega) \cap L^2_0(\Omega)$ be such that 
	\begin{equation}\label{eq:lemaux1}
		-(\nabla p_1,\nabla \tilde{p})_{L^2(\Omega)} = 
		  (\Delta u_{\mcF},\nabla \tilde{p})_{L^2(\Omega)} 
		+ (f_1,\nabla \tilde{p})_{L^2(\Omega)}
	\end{equation}
	holds for all $\tilde{p} \in  H^1(\Omega) \cap L^2_0(\Omega)$. 

	Note that $\Delta u_{\mcF} + f_1 \in [L^2(\Omega)]^2$ 
	and therefore the right-hand-side is a functional
	in $[H^1_0(\Omega)]^*$. So, existence and uniqueness 
	of $p_1 \in H^1(\Omega)\cap L^2_0(\Omega)$ is guaranteed. Using the choice $\tilde{p}:=p_1$, we obtain
	\begin{equation}\nonumber
		\| \nabla p_1\|_{L^2(\Omega)} \le \|\Delta u_{\mcF}\|_{L^2(\Omega)} 
		+ \|f_1\|_{L^2(\Omega)}
			\le \| u_{\mcF}\|_{H^2(\Omega)} + \|f_1\|_{L^2(\Omega)}.
	\end{equation}
	Using Poincar\'e's inequality, we obtain further
	\begin{equation}\label{eq:lemaux:est1}
		\| p_1\|_{H^1(\Omega)} \les \| u_{\mcF}\|_{H^2(\Omega)} + \|f_1\|_{L^2(\Omega)}.
	\end{equation}
	Let $p_2\in H^1(\Omega) \cap L^2_0(\Omega)$ be such that
	\begin{equation}\nonumber
		-(\nabla p_2, \nabla \tilde{p})_{L^2(\Omega)} 
		= - \beta ( u_{\mcF}, \nabla \tilde{p} )_{L^2(\Omega)} 
		+ ( f_2, \nabla \tilde{p})_{L^2(\Omega)}
	\end{equation}
	for all $\tilde{p} \in H^1(\Omega) \cap L^2_0(\Omega)$.
	This implies, as $u_{\mcF}\in [H^1_0(\Omega)]^2$
		and $f_2\in [H^1_0(\Omega)]^2$, that
	\begin{equation}\label{eq:lemaux2}
		-(\nabla p_2, \nabla \tilde{p})_{L^2(\Omega)} 
		= \beta (\nabla \cdot u_{\mcF}, \tilde{p} )_{L^2(\Omega)} 
		- (\nabla\cdot f_2, \tilde{p})_{L^2(\Omega)}
	\end{equation}
	holds. As $\beta \nabla\cdot u_{\mcF}-\nabla\cdot f_2\in L^2(\Omega)$,
	existence and uniqueness of $p_2$ is guaranteed.
	Condition~\citecond{(R1)} implies moreover $p_2\in H^2(\Omega)$ and
	\begin{equation}\label{eq:lemaux:est2}
		\|p_2\|_{H^2(\Omega)}^2 
		\les \beta^{2} \|\nabla \cdot u_{\mcF}\|_{L^2(\Omega)}^2
		  + \|\nabla\cdot f_2\|_{L^2(\Omega)}^2
		\les \beta^{2} \|u_{\mcF}\|_{H^1(\Omega)}^2 + \|f_2\|_{H^1(\Omega)}^2.
	\end{equation}
	Note that from~\eqref{eq:lemaux0}, \eqref{eq:lemaux1} 
	and \eqref{eq:lemaux2}, we obtain
	\begin{equation*}
		(\nabla (p_1+p_2),\nabla \tilde{p})_{L^2(\Omega)} =
			(\nabla p_{\mcF},\nabla \tilde{p})_{L^2(\Omega)}
	\end{equation*}
	is satisfied for all $\tilde{p} \in H^1(\Omega)\cap 
	L^2_0(\Omega)$, which implies (because $p_{\mcF}\in L^2_0(\Omega)$ and $p_1+p_2\in L^2_0(\Omega)$) 
	that $p_{\mcF}=p_1+p_2$ is satisfied.
	So, we have using~\eqref{eq:lemaux:est1} and~\eqref{eq:lemaux:est2}
	\begin{align*}
		&\|p_{\mcF}\|_{P_{+,k}}^2
		=h_k^2 \|p_{\mcF}\|_{H^1(\Omega)+\beta^{-1} H^2(\Omega)}^2 \\
			&\qquad \le \inf_{\substack{p=q_1+q_2,\\q_1\in H^1(\Omega)\cap L^2_0(\Omega)\\
					q_2\in H^2(\Omega) \cap L^2_0(\Omega)} } h_k^2 \|q_1\|_{H^1(\Omega)}^2 +h_k^2\|q_2\|_{\beta^{-1} H^2(\Omega)}^2 \\
			&\qquad \le h_k^2 \|p_1\|_{H^1(\Omega)}^2 +h_k^2\|p_2\|_{\beta^{-1} H^2(\Omega)}^2	\\
			&\qquad \les h_k^2 \|u_{\mcF}\|_{H^2(\Omega)}^2 +h_k^2\|u_{\mcF}\|_{\beta H^1(\Omega)}^2 +h_k^2\|f_1 \|_{L^2(\Omega)}^2+h_k^2\|f_2 \|_{\beta^{-1} H^1(\Omega)}^2\\
			&\qquad = h_k^2 \|u_{\mcF}\|_{H^2(\Omega)\cap \beta H^1(\Omega)}^2 +h_k^2 \|f_1 \|_{L^2(\Omega)}^2+h_k^2 \|f_2 \|_{\beta^{-1} H^1(\Omega)}^2.
	\end{align*}
	As $f_2\in [H^1_0(\Omega)]^2$ was chosen arbitrarily, we can take the infimum
	over all $f_2$, which finishes the proof.
\end{proof}

\def\proof{\emph{Proof of Lemma~\ref{lem:aux2}.}}
\begin{proof}
	As $H^1_0(\Omega)$ is dense in $L^2(\Omega)$, for $u\in [H^2(\Omega)]^2$ the function $-\Delta u \in [L^2(\Omega)]^2$
	can be approximated by some function $w^{\epsilon} \in [H^1_0(\Omega)]^2$ such that
	$	\|(-\Delta) u - w^{\epsilon}\|_{L^2(\Omega)}^2 \le \epsilon$.
	So, we can introduce an operator $-\Delta^{\epsilon}:[H^2(\Omega)]^2\rightarrow [H^1_0(\Omega)]^2$, which
	approximates $-\Delta u$ in $[H^1_0(\Omega)]^2$, satisfying the error estimate
	\begin{equation*}
		\|(-\Delta) u - (-\Delta^{\epsilon}) u \|_{L^2(\Omega)}^2 \le \epsilon
	\end{equation*}
	for all $u \in H^2(\Omega)$. So, $-\Delta^{\epsilon}u$ approaches $-\Delta u$
	in the $L^2$-sense for $\epsilon\rightarrow0$. (As $-\Delta u\not\in [H^1_0(\Omega)]^2$ in general,
	$\|-\Delta^{\epsilon} u\|_{H^1(\Omega)}$ is not convergent for $\epsilon\rightarrow0$.)
	
	Analogously, we introduce the operator $\nabla^{\epsilon}:H^1(\Omega)\rightarrow [H^1_0(\Omega)]^2$, which
	approximates $\nabla p$ in $[H^1_0(\Omega)]^2$, satisfying the error estimate
	\begin{equation*}
		\|\nabla p - \nabla^{\epsilon}p \|_{L^2(\Omega)}^2 \le \epsilon
	\end{equation*}
	for all $p\in H^1(\Omega)$. So, $\nabla^{\epsilon}p$ approaches $\nabla p$
	in the $L^2$-sense for $\epsilon\rightarrow0$. (As $\nabla p\not\in [H^1_0(\Omega)]^2$ in general,
	$\|\nabla^{\epsilon} p\|_{H^1(\Omega)}$ is not convergent for $\epsilon\rightarrow0$.)
	
	The idea of this proof is to show that for all $\epsilon>0$ there
	is some $\tilde{x}^{\epsilon}\in X$ such that
	\begin{align}
		\mcB(x_{\mcF},\tilde{x}^{\epsilon})&\ges h_k^{-2} \|u_{\mcF}\|_{U_{+,k}}^2 
			- \epsilon (1+\beta^{1/2}+\beta) h_k^{-1} \|x_{\mcF}\|_{X_{+,k}} -\epsilon^2 \mbox{ and} \label{eq:aux2a}\\
		\mcF(\tilde{x}^{\epsilon})&\les h_k^{-2} \|\mcF\|_{[X_{-,k}]^*} 
							\|x_{\mcF}\|_{X_{+,k}} + \epsilon (1+\beta^{1/2}) h_k^{-1} 
							\|\mcF\|_{(X_{-,k})^*}+\epsilon^2 \label{eq:aux2}
	\end{align}
	holds. As $\mcB(x_{\mcF},\tilde{x}^{\epsilon})=\mcF(\tilde{x}^{\epsilon})$, the statement of the Lemma follows for $\epsilon\rightarrow0$. 

	In the following, we show~\eqref{eq:aux2a} and~\eqref{eq:aux2} for
	$\tilde{x}^{\epsilon}:=(-\Delta^{\epsilon} u_{\mcF},\nabla\cdot\nabla^{\epsilon} p_{\mcF})$. 
	First, we show~\eqref{eq:aux2a}. Here, we estimate the summands of $\mcB(x_{\mcF},\tilde{x}^{\epsilon})$ separately. We have
	\begin{align*}
		&(\nabla u_{\mcF},\nabla (-\Delta^{\epsilon}) u_{\mcF})_{L^2(\Omega)} 
		 = (\Delta u_{\mcF}, \Delta^{\epsilon} u_{\mcF})_{L^2(\Omega)}  \\
		&\quad \ge  (\Delta u_{\mcF}, \Delta u_{\mcF})_{L^2(\Omega)} - \epsilon \|\Delta u_{\mcF}\|_{L^2(\Omega)}
		\ge  \|u_{\mcF}\|_{H^2(\Omega)}^2-\epsilon \beta^{1/2} h_k^{-1} \| x_{\mcF}\|_{X_{+,k}},
	\end{align*}
	where we use for the first step that $\Delta^{\epsilon}$ maps into $H^1_0(\Omega)$. For the second
	step, we use the Cauchy-Schwarz inequality and the upper bound for $\|-\Delta u_{\mcF}-(-\Delta^{\epsilon}) u_{\mcF}\|_{L^2(\Omega)}$.
	Friedrichs' inequality is used for the third step. Using similar arguments, we obtain for the
	next summand
	\begin{align*}
		&\beta(u_{\mcF},-\Delta^{\epsilon} u_{\mcF})_{L^2(\Omega)}	
			\ge \beta(u_{\mcF},-\Delta u_{\mcF})_{L^2(\Omega)} - \epsilon \beta\|u_{\mcF}\|_{L^2(\Omega)}\\
			&\qquad =  \beta(\nabla u_{\mcF},\nabla u_{\mcF})_{L^2(\Omega)} - \epsilon \beta\|u_{\mcF}\|_{L^2(\Omega)}
			\ges  \beta\|u_{\mcF}\|_{H^1(\Omega)}^2 - \epsilon \beta h_k^{-1} \|x_{\mcF}\|_{X_{+,k}}.
	\end{align*}
	We obtain for the last two summands, corresponding to the (2,1)- and the (1,2)-block, 
	again using similar arguments, including the bound for $\|\nabla p_{\mcF}-\nabla^{\epsilon} p_{\mcF}\|_{L^2(\Omega)}$, that
	\begin{equation*}
	\begin{aligned}
		&(\nabla\cdot u_{\mcF}, \nabla\cdot\nabla^{\epsilon} p_{\mcF})_{L^2(\Omega)} + (\nabla\cdot(-\Delta^{\epsilon}) u_{\mcF},p_{\mcF})_{L^2(\Omega)}\\  
			&\quad= -(\nabla \nabla\cdot u_{\mcF},\nabla^{\epsilon}p_{\mcF})_{L^2(\Omega)}+(\Delta^{\epsilon}u_{\mcF},\nabla p_{\mcF})_{L^2(\Omega)} \\
			&\quad\ge - (\nabla \nabla\cdot u_{\mcF},\nabla^{\epsilon}p_{\mcF})_{L^2(\Omega)} + (\Delta^{\epsilon}u_{\mcF},\nabla^{\epsilon} p_{\mcF})_{L^2(\Omega)}- \epsilon \|\Delta^{\epsilon} u_{\mcF}\|_{L^2(\Omega)} 
	\end{aligned}	
	\end{equation*}		
	holds. Using integration by parts, the Cauchy-Schwarz inequality and the upper bound
	for $\|-\Delta u_{\mcF}-(-\Delta^{\epsilon}) u_{\mcF}\|_{L^2(\Omega)}$, we obtain further
	\begin{equation*}
	\begin{aligned}
		&(\nabla\cdot u_{\mcF}, \nabla\cdot\nabla^{\epsilon} p_{\mcF})_{L^2(\Omega)} + (\nabla\cdot(-\Delta^{\epsilon}) u_{\mcF},p_{\mcF})_{L^2(\Omega)}\\  
			&\quad\ge (\nabla u_{\mcF},\nabla\nabla^{\epsilon}p_{\mcF})_{L^2} + (\Delta u_{\mcF},\nabla^{\epsilon} p_{\mcF})_{L^2}- \epsilon (\|\Delta u_{\mcF}\|_{L^2} +\|\nabla^{\epsilon} p_{\mcF}\|_{L^2}+\epsilon) \\
			&\quad= \underbrace{(\nabla u_{\mcF},\nabla\nabla^{\epsilon}p_{\mcF})_{L^2} - (\nabla u_{\mcF},\nabla \nabla^{\epsilon} p_{\mcF})_{L^2}}_{\displaystyle =0} - \epsilon (\|\Delta u_{\mcF}\|_{L^2} +\|\nabla^{\epsilon} p_{\mcF}\|_{L^2}+\epsilon).
	\end{aligned}	
	\end{equation*}					
	Finally, we obtain using the upper bound for $\|\nabla p_{\mcF}-\nabla^{\epsilon} p_{\mcF}\|_{L^2(\Omega)}$ and the definition of $\|\cdot\|_{X_{+,k}}$
	that
	\begin{equation*}
	\begin{aligned}
			&(\nabla\cdot u_{\mcF}, \nabla\cdot\nabla^{\epsilon} p_{\mcF})_{L^2(\Omega)} + (\nabla\cdot(-\Delta^{\epsilon}) u_{\mcF},p_{\mcF})_{L^2(\Omega)}\\  	
			&\quad\ge - \epsilon (\|\Delta u_{\mcF}\|_{L^2(\Omega)} +\|\nabla p_{\mcF}\|_{L^2(\Omega)}+\epsilon)\ge  -\epsilon (\|u_{\mcF}\|_{H^2(\Omega)} +\|p_{\mcF}\|_{H^1(\Omega)}+\epsilon)\\	
			&\quad\ge -  \epsilon (1+\beta^{1/2}) h_k^{-1} \|x_{\mcF}\|_{X_{+,k}} - \epsilon^2
	\end{aligned}
	\end{equation*}
	holds. So, we have shown~\eqref{eq:aux2a}.
	The next step is to show~\eqref{eq:aux2}. Let $\mcF(\tilde{u},\tilde{p}):= (f,\tilde{u})_{L^2(\Omega)}+(g,\tilde{p})_{L^2(\Omega)}$
	for $f \in [L^2(\Omega)]^2$ and $g \in H^1_0(\Omega)\cap L^2_0(\Omega)$. 
	Let $f_2\in [H^1_0(\Omega)]^2$ and $f_1:=f-f_2$. Using the same arguments as above,
	\begin{equation*}
		(f_1,-\Delta^{\epsilon} u_{\mcF})_{L^2(\Omega)}
			\les
			\| f_1\|_{L^2(\Omega)} \|u_{\mcF}\|_{H^2(\Omega)} + \epsilon\| f_1\|_{ L^2(\Omega)}
	\end{equation*}
	holds as well as
	\begin{align*}
		&(f_2,-\Delta^{\epsilon} u_{\mcF})_{L^2(\Omega)}
		\les (\nabla f_2,\nabla u_{\mcF})_{L^2(\Omega)}+ \epsilon\| f_2\|_{L^2(\Omega)}\\
		&\qquad	\les
			\| f_2\|_{\beta^{-1} H^1(\Omega)} \|u_{\mcF}\|_{\beta H^1(\Omega)} + \epsilon\beta^{1/2} \| f_2\|_{\beta^{-1} H^1(\Omega)}.
	\end{align*}
	This implies
	\begin{align*}
		&(f,-\Delta^{\epsilon} u_{\mcF})_{L^2(\Omega)}\\
		&\qquad\les \|f\|_{L^2(\Omega)+\beta^{-1} H^1_0(\Omega)} 
			\|u_{\mcF}\|_{H^2(\Omega)\cap \beta H^1(\Omega)} + \epsilon (1+\beta^{1/2})
				\|f\|_{L^2(\Omega)+\beta^{-1} H^1_0(\Omega)}\\
		&\qquad\les \|f\|_{L^2(\Omega)+\beta^{-1} H^1_0(\Omega)} 
			\|u_{\mcF}\|_{H^2(\Omega)\cap \beta H^1(\Omega)} + \epsilon h_k^{-1} (1+\beta^{1/2})
				\|\mcF\|_{[X_{-,k}]^*}.
	\end{align*}
	Let $p_2\in H^2(\Omega)$ and $p_1:=p-p_2$.
	We have
	\begin{equation*}
		(g,\nabla\cdot\nabla^{\epsilon} p_1)_{L^2(\Omega)} =
			-(\nabla g,\nabla^{\epsilon} p_1)_{L^2(\Omega)}
			\les \|g\|_{H^1(\Omega) }
					\|p_1\|_{ H^1(\Omega) } + \epsilon \|g\|_{H^1(\Omega)}.
	\end{equation*}
	Using $g\in H^1_0(\Omega)$, we have also
	\begin{align*}
		&(g,\nabla\cdot\nabla^{\epsilon} p_2)_{L^2(\Omega)} 
		=	- (\nabla g,\nabla^{\epsilon} p_2)_{L^2(\Omega)} 
		\les	-(\nabla g,\nabla p_2)_{L^2(\Omega)} + \epsilon \|g\|_{H^1(\Omega)}\\
		&\quad =( g,\nabla \cdot \nabla p_2)_{L^2(\Omega)} + \epsilon \|g\|_{H^1(\Omega)}
			\le \|g\|_{ \beta L^2(\Omega) }
					\|p_1\|_{ \beta^{-1} H^2(\Omega) } + \epsilon \|g\|_{H^1(\Omega)}
	\end{align*}
	and, by combining these estimates,	
	\begin{align*}
		(g,\nabla\cdot\nabla^{\epsilon} p_{\mcF})_{L^2(\Omega)}
			&\les \| g\|_{H^1_0(\Omega)+\beta L^2_0(\Omega)}
			\|p_{\mcF}\|_{H^1(\Omega)+ \beta^{-1} H^2(\Omega)} + \epsilon \|g\|_{H^1(\Omega)}\\
			&\les \| g\|_{H^1_0(\Omega)+\beta L^2_0(\Omega)}
			\|p_{\mcF}\|_{H^1(\Omega) + \beta^{-1} H^2(\Omega)} + \epsilon h_k^{-1}
				\|\mcF\|_{(X_{-,k})^*}.
	\end{align*}
	By combining these results, we immediately obtain~\eqref{eq:aux2}. This
	finishes the proof.
\end{proof}

\def\proof{\emph{Proof of inequality~\eqref{lem:equiv:aux2}.}}
\begin{proof}
	First note that
	\begin{align*}
		&\|\psi_k p_k\|_{H^1(\Omega)} \eqsim \|\nabla (\psi_k p_k) \|_{L^2(\Omega)} + \|\psi_k p_k\|_{L^2(\Omega)}\\
		&\qquad	 \les \|(\nabla \psi_k) p_k \|_{L^2(\Omega)} +  \|\psi_k \nabla p_k \|_{L^2(\Omega)} + \|\psi_kp_k\|_{L^2(\Omega)}\\
			&\qquad \le \|\nabla \psi_k\|_{L^{\infty}(\Omega)} \| p_k \|_{L^2(\Omega)} 
								+ \| \psi_k\|_{L^{\infty}(\Omega)} \|\nabla p_k \|_{L^2(\Omega)} 
								+ \| \psi_k\|_{L^{\infty}(\Omega)} \|p_k\|_{L^2(\Omega)}
	\end{align*}
	holds due to the product rule and H\"older's inequality. Here, $\|\cdot\|_{L^{\infty}(\Omega)}$ is the
	standard $L^{\infty}$-norm (essential supremum). As $\psi_k(\xi)$ is piecewise linear and 
	$\psi_k(\xi)\in[0,1]$, $|\nabla \psi_k(\xi)|\le h_k^{-1}$ holds for almost all $\xi\in \Omega$. Using this, $|\psi_k(\xi)|\le1$
	and a standard inverse inequality for estimating $\|\nabla p_k \|_{L^2(\Omega)}$ from above, one obtains
		$\|\psi_k p_k\|_{H^1(\Omega)} \les (1+h_k^{-1}) \|p_k\|_{L^2(\Omega)}$
	and, assuming $h_k \les 1$, the desired result.
\end{proof}

\def\proof{\emph{Proof of inequality~\eqref{eq:kinv}.}}
\begin{proof}
	Assume that the mesh consists of the elements $\mathcal{T}_i$ for $i=1,\ldots,M_k$.
	Note that
	it suffices to show~\eqref{eq:kinv} for each element $\mathcal{T}_i$ or that
	\begin{equation}\label{eq:kinv2}
		(p_k,p_k)_{L^2(\mathcal{T}_i)}\le C (p_k,\psi_k p_k)_{L^2(\mathcal{T}_i)}
	\end{equation}
	for all $p_k\in \hat{P}_k$.
	Note that $\psi_k = 1$ on all interior elements (no vertex is located on $\partial\Omega$), so~\eqref{eq:kinv2} 
	holds with $C=1$ in this case. So, it remains to show the inequality for the elements
	elements where one (case~1) or two vertices (case~2) are located on $\partial\Omega$. (Note that by
	assumption each element has at least one vertex in the interior of $\Omega$).

	Note that, both $p_k$ and $\psi_k$, are linear functions.
	The integrals in~\eqref{eq:kinv2} can be computed using the reference element 
	$\Delta:=\{(\xi_1,\xi_2) \in (0,1)^2\;:\; \xi_1+\xi_2 < 1 \}$. Using the linear transformation
	$\Phi_i:\Delta\rightarrow \mathcal{T}_i$, the substitution rule states that~\eqref{eq:kinv2}
	is equivalent to
	\begin{equation}\label{eq:kinv3b}
	  |\mbox{det} \nabla \Phi_i |(\hat{p}_k,\hat{p}_k)_{L^2(\mathcal{T}_i)}
	  	\le C |\mbox{det} \nabla \Phi_i |(\hat{p}_k,\hat{\psi}_k \hat{p}_k)_{L^2(\mathcal{T}_i)},
	\end{equation}
	where $\hat{p}(\xi) = u(\Phi_i^{-1}(\xi))$ and $\hat{\psi}(\xi) = \psi(\Phi_i^{-1}(\xi))$.
	Here, $|\mbox{det} \nabla \Phi_i |$ is the absolute value of the Jacobi determinant of the transformation. As this
	is a positive constant, it can be canceled out.

	For the \emph{case~1}, the function $\hat{\psi}_k$ takes the value $1$
	on two vertices and the value $0$ on one vertex. We assume without loss of generality, that it takes the value $0$
	on $(0,1)$. So, we obtain $\hat{\psi}_k(\xi)=1-\xi_2$.
	Using the ansatz $\hat{p}_k(\xi_1,\xi_2):=a_0+a_1 \xi_1+a_2 \xi_2$,
	we can compute both integrals in~\eqref{eq:kinv3b} and obtain that~\eqref{eq:kinv3b} holds for $C=\frac13\left(6+\sqrt{6}\right)$.

	For the \emph{case~2}, the function $\hat{\psi}_k$ takes the value $0$
	on two vertices and the value $1$ on one vertex. We assume without loss of generality that it
	takes the value $1$ on the vertex $(0,1)$. So, we obtain 
	$\hat{\psi}_k(\xi)=\xi_2$. Here we obtain that~\eqref{eq:kinv3b}
	holds for $C=\left(4+\sqrt{6}\right)$. This finishes the proof.
\end{proof}
\def\proof{\emph{Proof.}}

\bibliographystyle{amsplain}

\providecommand{\bysame}{\leavevmode\hbox to3em{\hrulefill}\thinspace}
\providecommand{\MR}{\relax\ifhmode\unskip\space\fi MR }
\providecommand{\MRhref}[2]{%
  \href{http://www.ams.org/mathscinet-getitem?mr=#1}{#2}
}
\providecommand{\href}[2]{#2}

\mbox{}

\begin{center}
		First Published in SINUM in  53 (6), p. 634 - 2654, 2015,
		published by the Society of Industrial and Applied Mathematics (SIAM)\\
		\url{http://epubs.siam.org/doi/10.1137/140969658} \\
		Copyright by SIAM. Unauthorized reproduction of this article is  prohibited.
\end{center}

\end{document}